\documentclass[parskip=half,bibliography=totoc]{scrartcl}

\usepackage{comment}
\usepackage[automark]{scrlayer-scrpage}								
\usepackage{amsfonts}												
\usepackage{amsmath}												
\usepackage{mathtools}												
\usepackage{stmaryrd}												
\usepackage{amssymb}												
\usepackage{mathrsfs}												
\usepackage{accents}												
\usepackage[T1]{fontenc}											
\usepackage[utf8]{inputenc}											
\usepackage[top=1in,bottom=1in,left=1.25in,right=1.25in]{geometry}	
\usepackage{enumerate} 												
\usepackage{authblk}												
\usepackage{abstract}												
\usepackage{etoolbox}												
\usepackage[normalem]{ulem}

\usepackage{tikz-cd}												

\usepackage{amsthm}													

\usepackage[colorlinks=true,citecolor=blue,allcolors=blue]{hyperref}				
\usepackage[noabbrev]{cleveref}										

\newcommand{\p}{\partial}

\renewcommand{\d}{\mathrm{d}}

\newcommand{\ric}{\mathrm{Ric}}

\newcommand{\CS}{\mathrm{CS}}

\newcommand{\fre}{\mathfrak{e}}

\DeclareMathOperator{\tr}{tr}

\newtheoremstyle{mythm}
{}
{}
{\slshape}
{}
{\bfseries\sffamily}
{.}
{ }
{}
\newtheoremstyle{mydef}
{}
{}
{}
{}
{\bfseries\sffamily}
{.}
{ }
{}

\theoremstyle{mythm}
\newtheorem{thm}{Theorem}[section]

\newtheorem{prop}[thm]{Proposition}
\newtheorem{cor}[thm]{Corollary}
\newtheorem{lem}[thm]{Lemma}
\theoremstyle{mydef}
\newtheorem{mydef}[thm]{Definition}
\newtheorem{rem}[thm]{Remark}

\clearpairofpagestyles
\ihead[]{\headmark}
\ohead[]{\pagemark}
\deffootnote[1em]{0em}{1em}{%
	\textsuperscript{\thefootnotemark}%
}
\setfootnoterule{3em}

\apptocmd{\sloppy}{\hbadness 10000\relax}{}{}

\title{Heisenberg-invariant self-dual Einstein manifolds}
\author[1]{V.\ Cort\'es}
\author[1,2]{\'A. Murcia}

\affil[1]{\normalsize Department of Mathematics\hfill \protect\\  University of Hamburg\hfill \protect\\  
Bundesstra\ss e 55, D-20146 Hamburg, Germany\hfill \protect\\  
\texttt{vicente.cortes@uni-hamburg.de}  \vspace{1em} }

\affil[2]{\normalsize Instituto de F\'isica Te\'orica UAM-CSIC\hfill \protect\\  
C/ Nicol\'as Cabrera 13-15, C.U. Cantoblanco, E-28049 Madrid, Spain\hfill \protect\\  
\texttt{angel.murcia@csic.es}
}

\date{}

\begin{document}
\maketitle

\begin{abstract}
We classify all self-dual Einstein four-manifolds invariant under a principal 
action of the three-dimensional Heisenberg group with non-degenerate orbits. The metrics are explicit and we find, in particular, that the Einstein constant can take any value. Then we study when the corresponding (Riemannian or neutral-signature) metrics are (geodesically) complete. Finally, we exhibit the solutions of non-zero Ricci-curvature as different branches of one-loop deformed universal hypermultiplets in Riemannian and neutral signature.

\end{abstract}

\clearpage
\setcounter{tocdepth}{1}
\tableofcontents

\section{Introduction}

Principal group actions on pseudo-Riemannian manifolds play a prominent role in differential geometry. Indeed, many fundamental concepts, such as principal bundles or homogeneous spaces, are based on this notion. Among them, it is particularly interesting to consider cohomogeneity one principal actions, in which the corresponding orbits are of codimension one. This allows to reduce interesting systems of partial differential equations, such as the Einstein equation, to systems of ordinary differential equations. This is extremely relevant in theoretical physics as well, since the above principle allows to solve the field equations of general relativity in many important cases, for instance, within the context of cosmological models \cite{EM}.

On the other hand, the homogeneous quaternionic K\"ahler manifolds of negative scalar curvature (except the quaternionic hyperbolic spaces) have been shown to admit  a canonical deformation to a complete quaternionic K\"ahler manifold with a cohomogeneity-one isometric action \cite{CST}. This deformation is  a particular case of what is called one-loop deformation \cite{AMTV,RSV}, which appears in the study of string theory and supergravity with one-loop corrections. In four dimensions it is found \cite{CST} that the isometry group of the deformed quaternionic K\"ahler manifold is $\mathrm{O}(2) \ltimes \mathrm{H}$, with $\mathrm{H}$ the three-dimensional Heisenberg group, and this motivated to carry out the classification of all Riemannian Einstein metrics of non-positive scalar curvature  which are invariant under the action of $\mathrm{SO}(2) \ltimes \mathrm{H}$ in $\mathbb{R}^4$ \cite{CS}. Apart of a wealth of incomplete metrics, the authors of \cite{CS} showed that the only complete manifolds in the above class are the complex hyperbolic plane (also know as universal hypermultiplet in the physics literature) and its complete one-loop deformation. 

Recall that a four-dimensional Riemannian metric is called quaternionic-K\"ahler if it is Einstein and self-dual (for appropriate choice of orientation).  The notion of self-duality is also meaningful for metrics of neutral signature\footnote{But not for Lorentzian signature, cf.\ Remark \ref{rem:lorsollight}.} and self-dual Einstein metrics of neutral signature are also called quaternionic paraK\"ahler.

In this context, the objective of this work is the classification of all self-dual pseudo-Riemannian Einstein four-manifolds which admit a principal (cohomogeneity-one) isometric action of the Heisenberg group. 
The hypotheses are, on the one hand, more general than those of \cite{CS} in that we allow indefinite metrics and assume only the symmetry group $\mathrm{H}$ rather than $\mathrm{SO}(2) \ltimes \mathrm{H}$ but, on the other hand, more specific as we restrict to self-dual metrics. 

Up to an overall sign in the metric, the manifolds considered can be decomposed as $(\mathcal{I}\times \mathrm{H}, \varepsilon \d t^2+ \chi_t)$, being $\mathcal{I} \subset \mathbb{R}$ an open interval parametrized by $t$ (that we call time), $\varepsilon=\pm 1$ and $\{\chi_t\}_{t \in \mathcal{I}}$ a family of Riemannian 
or Lorentzian 
metrics on $\mathrm{H}$, respectively. 
The key feature that will allow us to perform such classification is the use of a family of orthonormal or Witt frames on $\mathrm{H}$, which are interpreted as the time-evolution of an initial (orthonormal or Witt) frame on $\mathrm{H}$. 

More concretely, on the one hand we first consider the proper (i.e.\ $\mathrm{scal} \neq 0$) quaternionic (para)K\"ahler four-manifolds with a Heisenberg principal group action and, attending to the 
causal character of the center of the Heisenberg group (for neutral-signature manifolds), we determine completely their isometry type. In the Riemannian case, apart from encountering the complex hyperbolic metric and the complete and incomplete one-loop deformed universal hypermultiplet metrics of negative scalar curvature as reported in \cite{CS}, we also find a counterpart of positive scalar curvature.

For neutral-signature metrics, we are able to identify the  solutions of negative scalar curvature as quaternionic paraK\"ahler geometries arising from the so-called temporal and Euclidean supergravity $c$-maps\footnote{We remind that the temporal (respectively, Euclidean) supergravity $c$-map is induced by the reduction of four-dimensional Minkowskian (respectively, Euclidean) $\mathcal{N}=2$ supergravity coupled to vector multiplets over a timelike (respectively, spacelike) dimension, while the one-loop deformed universal hypermultiplet  metric arises from the usual supergravity $c$-map induced by the reduction of four-dimensional $\mathcal{N}=2$ supergravity coupled to vector multiplets over a spacelike dimension \cite{CMMS1,CMMS2}.} while those with positive scalar curvature do not seem to have been previously considered.  Furthermore, we study when such metrics are complete. 

On the other hand, we also investigate (para)hyperK\"ahler four-manifolds endowed with a principal action of the Heisenberg group and provide a classification of all of them in terms of their isometry type, in similar lines to the quaternionic case. It turns out that the Ricci-flat examples are incomplete with exception of a class of flat examples of neutral signature.

The outline of the article is as follows. In Section \ref{sec:2},  we 
introduce some concepts and nomenclature which will be of use along the document. Before imposing any differential equation, 
we begin by specifying precisely the manifolds with Heisenberg symmetry and the general form of the invariant metric to be considered, see Proposition~\ref{prop:1stNF} for an 
adapted frame. Then, using part of the Einstein equation, 
we show in Proposition \ref{prop:zcausalconst} that the causal character of the Heisenberg center is time-independent and in Proposition~\ref{prop:2ndNF} that the adapted frame can be considerably simplified.

Afterwards, in Section \ref{sec:3}  we determine all quaternionic (para)K\"ahler four-manifolds which admit an isometric  principal action of the Heisenberg group  with non-degenerate orbits. In the neutral signature case, these (connected) manifolds are referred to as \emph{timelike, spacelike or lightlike} quaternionic (para)K\"ahler Heisenberg four-manifolds, depending on the causal character of the Heisenberg center. 

In the Riemannian case, as well as in the case of timelike or spacelike Heisenberg center, we reduce the problem to a system ordinary differential equations  of first order for two functions $a$ and $b$, 
see Propositions~\ref{prop:edosqk} and \ref{prop:eqsspace}. The resulting quaternionic\linebreak[3] K\"ahler manifolds prove to be conformally K\"ahler\footnote{The Hermitian structures encountered in Proposition~\ref{Weyl:prop} appear naturally from the perspective of the one-loop deformed c-map. The investigation of such structures in arbitrary dimensions is part of an ongoing project of V.C.\ joint with Arpan Saha and Danu Thung.} (see Proposition~\ref{Weyl:prop})   whereas the resulting quaternionic paraK\"ahler manifolds are conformally K\"ahler or conformally paraK\"ahler, depending on the causal character of the Heisenberg center (see Proposition~\ref{prop:spaceconfk}). 

It turns out that the ode system takes a particularly nice form when  $\varepsilon \Lambda = -6 k^2$,
where $k$ is the structure constant of the adapted frame at initial time, up to a numerical factor. We refer to this case as the stationary case, as its solutions are stationary in the sense that the logarithmic derivatives of the unknown functions are constant. We determine all maximal stationary solutions and show that they define homogeneous spaces, see Propositions~\ref{prop:statsol} and \ref{prop:statspace}. More precisely, each of these homogeneous spaces can be realized as an open orbit of a four-dimensional solvable Lie group acting by isometries on a symmetric space.  Altogether, we are able to state the following:
\begin{thm}
\label{thm:introstat}
Let $(M,g)$ be a stationary (timelike) quaternionic (para)K\"ahler Heisenberg four-manifold. Then there exists an $\mathrm{H}$-equivariant 
diffeomorphism $M\cong \mathcal{I}\times \mathrm{H}$ such that the metric takes the following form: 
\begin{equation*}
g=\varepsilon \d t^2 + \varepsilon e^{4k (t-t_0)} \fre^1_{t_0} \otimes \fre^1_{t_0}+e^{2k (t-t_0)}\left ( \fre^2_{t_0} \otimes \fre^2_{t_0}+\fre^3_{t_0} \otimes \fre^3_{t_0}  \right)\,,
\end{equation*}
where $(\fre^i_{t_0})$ denotes a left-invariant coframe of the Heisenberg group $\mathrm{H}$ whose dual frame $(\fre_i^{t_0})$ is such that $[\fre_2^{t_0},\fre_3^{t_0}]=-2k \fre_1^{t_0}$ with $k \neq 0$. 

If $\mathcal{I}=\mathbb{R}$, then $(M,g)$ is isometric to an open orbit of the solvable Iwasawa subgroup of $\mathrm{SU} (1,2)\cong \mathrm{SU}(2,1)$ on the symmetric space 
\begin{equation*} \dfrac{\mathrm{SU}\left (\dfrac{3+\varepsilon}{2},\dfrac{3-\varepsilon}{2}\right )}{\mathrm{S}\left (\mathrm{U}(1) \times \mathrm{U}\left (\dfrac{3+\varepsilon}{2},\dfrac{1-\varepsilon}{2}\right )\right )},\end{equation*} where $\mathrm{U}(p,q)$ denotes the (pseudo-)unitary group 
of the Hermitian sesquilinear form of index $q$. Moreover, when $(M,g)$ is Riemannian (resp. neutral signature) it is complete (resp. incomplete). 

Similarly, if $(M,g)$ is a stationary spacelike quaternionic paraK\"ahler Heisenberg four-manifold, then the metric takes the form:
\begin{equation*}
g=-\d t^2 +e^{-2k (t-t_0)}\left ( -\fre^1_{t_0} \otimes \fre^1_{t_0}+\fre^2_{t_0} \otimes \fre^2_{t_0}  \right)+ e^{-4k (t-t_0)} \fre^3_{t_0} \otimes \fre^3_{t_0}\,,
\end{equation*}
where now $[\fre_1^{t_0}, \fre_2^{t_0}]=-2 k \fre_3^{t_0}$.

If $\mathcal{I}=\mathbb{R}$, these solutions are isometric to an open orbit of a four-dimensional solvable subgroup (which contains a Heisenberg subgroup) of $\mathrm{SL}(3,\mathbb{R})$ on the symmetric space:
\begin{equation*}
\frac{\mathrm{SL}(3,\mathbb{R})}{\mathrm{S}(\mathrm{GL}(1,\mathbb{R})\times \mathrm{GL}(2,\mathbb{R}))}\,.
\end{equation*}
Furthermore, they all are incomplete.
\end{thm}
In the non-stationary case, we show that the above ode system can be explicitly and completely solved, see Propositions~\ref{prop:neg(time)q(para)}, \ref{prop:pos(time)q(para)}, \ref{neg_space_q_para} and \ref{pos_space_q_para}. We find that the solutions occur in one-parametric families and that the parameter can be identified with the one-loop parameter in the perturbative quantum correction of the supergravity c-map and its temporal and Euclidean versions. Setting the parameter to zero (whenever possible on the considered branch) results in one of the stationary solutions. Geometrically this corresponds to a deformation of the locally symmetric space in the class of quaternionic (para)K\"ahler manifolds. In the case of lightlike center,  we find analogously all solutions to the aforementioned system of odes and observe that they are conformally flat, see Proposition~\ref{prop:light}. In particular, we are able to prove the following results. 

\begin{thm}
Let $(M,g)$ be a (timelike) quaternionic (para)K\"ahler Heisenberg four-manifold. Then the metric takes the following form:
\begin{equation*}
g=-\frac{3}{2 \varepsilon \Lambda  \rho^2} \frac{\rho+2\gamma}{\rho+\gamma} \left ( \varepsilon\d \rho^2+\frac{\varepsilon}{k^2} \left( \frac{\rho+\gamma}{\rho+2\gamma}\right)^2 \fre_{t_0}^1\otimes  \fre_{t_0}^1+2(\rho+\gamma) \left (\fre_{t_0}^2\otimes  \fre_{t_0}^2+\fre_{t_0}^3\otimes  \fre_{t_0}^3 \right)\right)\,,
\end{equation*}
where $(\fre^i_{t_0})$ denotes a left-invariant coframe of the Heisenberg group $\mathrm{H}$ such that $[\fre_2^{t_0},\fre_3^{t_0}]=-2k \fre_1^{t_0}$ with $k \neq 0$ and where $\rho \in \mathcal{I}$ for certain open interval $\mathcal{I} \subseteq \mathbb{R}$ specified as follows. We define the real numbers
\begin{equation*}
\begin{split}
\rho_l&=-\frac{e^{-(4-2l)i \pi/3}}{2k (\varepsilon \Lambda)^{1/3}(9k+\sqrt{81k^2+ \varepsilon\Lambda})^{1/3}}+e^{(4-2l)i \pi/3} \frac{(9k+\sqrt{81k^2+ \varepsilon\Lambda})^{1/3}}{2k (\varepsilon\Lambda)^{2/3}}\,, \, \, l=1,2,3\\ 
\rho_4&=\frac{\varepsilon \Lambda }{k  (-9 k \Lambda^4+\sqrt{ \Lambda^8 (81 k^2+\varepsilon\Lambda)})^{1/3}} -k \frac{1}{2 \Lambda^2} (-9 k \Lambda^4+\sqrt{\Lambda^8(81 k^2+\varepsilon\Lambda) }  )^{1/3}\,.
\end{split}
\end{equation*}
For any element $\rho_0$ of the set 
$$ \{\rho \in \mathbb{R}\, |\,  \rho \neq0\, , \rho+\gamma >0 \, \, \text{and}\, \, (-1)^{j+1}(\rho+2\gamma)>0\}\,,$$
we denote by $I_{\rho_0}^j$ the connected component containing $\rho_0$,  
where $2\gamma=-\rho_0\left (1+\frac{\varepsilon\Lambda}{3}\rho_0 \right )$ and where $\varepsilon\Lambda \neq 0$ is the Einstein constant. Then the following holds. 
\begin{enumerate}
\item Setting $\mathcal{I}= I^1_{\rho_l}$, $l=1,2,3$, we obtain all (timelike) quaternionic (para)K\"ahler Heisenberg four-manifolds with $\varepsilon\Lambda<0$. In particular, the solution given by $\rho_1$ is defined for all $\varepsilon\Lambda < 0$ while the other two are defined for all $\varepsilon\Lambda \leq -81k^2$ . The corresponding pseudo-Riemannian manifolds arising from these last two cases, together with those stemming from the first one for $\varepsilon\Lambda > -6k^2$, are incomplete while the first for $\varepsilon=1$ and $\Lambda < 6k^2$ is complete. 
\item Setting $\mathcal{I}= I^2_{\rho_4}$, we find  all (timelike) quaternionic (para)K\"ahler Heisenberg four-manifolds with $\varepsilon\Lambda>0$. They all are incomplete. 
\end{enumerate}
\label{thm:alltimeqknonstat}
\end{thm}

\begin{thm}
There exists a one-to-one correspondence between spacelike and timelike quaternionic paraK\"ahler Heisenberg four-manifolds. Any spacelike quaternionic paraK\"ahler Heisenberg four-manifold is isometric to:
\begin{equation*}
g=\frac{3}{2 \Lambda \rho^2} \frac{\rho+2\gamma}{\rho+\gamma} \left (-\d \rho^2+2(\rho+\gamma)\left (-\fre_{t_0}^1\otimes  \fre_{t_0}^1+\fre_{t_0}^2\otimes  \fre_{t_0}^2 \right)+ \left( \frac{\rho+\gamma}{k(\rho+2\gamma)}\right)^2 \fre_{t_0}^3\otimes  \fre_{t_0}^3\right)\,,
\end{equation*}
where $\rho$ is defined in the appropriate intervals $I^1_{\rho_l}$, $l=1,2,3$ or $I^2_{\rho_4}$, like in the timelike case, and where $\Lambda$ is the Einstein constant. 
Here $(\fre^i_{t_0})$ denotes a left-invariant coframe of  $\mathrm{H}$  such that $[\fre_1^{t_0},\fre_2^{t_0}]=-2k \fre_3^{t_0}$, $k \neq 0$. 
The (in)completeness properties are the same as those of their timelike analogues.
\label{thm:alllspaceqknonstat}
\end{thm}

\begin{thm}
All lightlike quaternionic paraK\"ahler Heisenberg four-manifolds $(M,g)$ are given by:
\begin{equation*}
g=-\d t^2+ e^{-\frac{2 \sqrt{\Lambda}}{\sqrt{3}}(t-t_0)} \fre_{t_0}^u \odot \fre_{t_0}^v +e^{-\frac{2 \sqrt{\Lambda}}{\sqrt{3}}(t-t_0)} \fre_{t_0}^3 \otimes \fre_{t_0}^3\,,
\end{equation*}
where $t \in \mathbb{R}$ and $\Lambda >0$ is the Einstein constant.  Here $(\fre_{t_0}^i)$  denotes a left-invariant (Witt) coframe on $\mathrm{H}$ such that $[\fre_v^{t_0}, \fre^{t_0}_3]=-2k \fre_u^{t_0}$, $k \neq 0$. The above manifolds are all conformally flat and incomplete.
\label{thm:alllightqknonstat}
\end{thm} 

Finally,  in Section \ref{sec:4} we classify all (para)hyperK\"ahler four-manifolds with an isometric principal action of the Heisenberg group with non-degenerate orbits.  We obtain the following classification result:
\begin{thm}
All (timelike) (para)hyperKähler Heisenberg four-manifolds are incomplete and isometric to:
\begin{equation*}
g=\varepsilon \d t^2+\frac{\varepsilon}{(1+3k (t-t_0))^{2/3}} \fre_{t_0}^1 \otimes \fre_{t_0}^1+ (1+3k (t-t_0))^{2/3}(\fre_{t_0}^2 \otimes \fre_{t_0}^2+\fre_{t_0}^3 \otimes \fre_{t_0}^3)\, , 
\end{equation*}
where $(\fre^i_{t_0})$ is a left-invariant coframe of  $\mathrm{H}$ such that $[\fre_2^{t_0},\fre_3^{t_0}]=-2k \fre_1^{t_0}$ for $k\neq 0$.  

All spacelike parahyperKähler Heisenberg four-manifolds are incomplete and isometric to:
\begin{equation*}
g=- \d t^2+ (1-3k (t-t_0))^{2/3}(-\fre_{t_0}^1 \otimes \fre_{t_0}^1+\fre_{t_0}^2 \otimes \fre_{t_0}^2)+\frac{1}{(1-3k (t-t_0))^{2/3}} \fre_{t_0}^3 \otimes \fre_{t_0}^3\, , 
\end{equation*}
where now $[\fre_1^{t_0},\fre_2^{t_0}]=-2k \fre_3^{t_0}$. 

Finally, all lightlike parahyperKähler Heisenberg four-manifolds are isometric to flat space $(\mathbb{R}^4,\eta)$.
\label{thm:allhk}
\end{thm}

{\bfseries Acknowledgements}

This work was supported by the German Science Foundation (DFG) under Germany's Excellence Strategy  --  EXC 2121 ``Quantum Universe'' -- 390833306. \'AM acknowledges additional support from the Deutscher Akademischer Austauschdienst (DAAD) through the Short-Term Research Grant No. 91791300 and from the Spanish FPU Grant No. FPU17/04964. 

We thank Carlos Shahbazi for useful comments,  Maciej Dunajski for discussion of the conformal rescaling 
in Remark \ref{conf:Rem} and Vestislav Apostolov and Claude LeBrun for pointing out the fact that a self-dual Riemannian Einstein metric
admits a local conformal rescaling to a K\"ahler metric if its Weyl tensor has a double eigenvalue. V.C.\ would also like to thank
Arpan Saha and Danu Thung for discussions about Hermitian structures on one-loop deformed c-map spaces in the context 
of a collaboration extending some of the results of this paper to higher dimensions.

\noindent 
Most of the computations presented in the document have been carried out with the help of \emph{Mathematica} 12.3, using the license of the University of Hamburg. Please find in the References section \cite{M} a link to download the \emph{Mathematica} file we have used.



\section{The Heisenberg group and Heisenberg four-manifolds}

\label{sec:2}

In this section we  revise basic features concerning left-invariant pseudo-Riemannian metrics on the Heisenberg group,   introduce the concept of 
a Heisenberg four-manifold and prove some preliminary results about Einstein Heisenberg manifolds.

\subsection{Heisenberg group}
Recall that the three-dimensional Heisenberg group $\mathrm{H}$ is the unique, up to isomorphism, connected and simply connected non-abelian nilpotent three-dimensional Lie group. Its Lie algebra $\mathfrak h$ is called the Heisenberg algebra. As for any Lie group, there is a natural bijection between left-invariant pseudo-Riemannian metrics on $\mathrm{H}$ and pseudo-Euclidean scalar products on $\mathfrak{h}$.
We will thus often refer to such a scalar product $\chi$ as a pseudo-Riemannian metric on $\mathfrak h$ and to 
the pair $(\mathfrak h, \chi)$ as a \emph{pseudo-Riemannian Heisenberg algebra.}

Let  $\chi$ be a pseudo-Riemannian metric on $\mathfrak{h}$. Then there always exists an orthonormal or a Witt basis\footnote{If $(V,\chi)$ is a three-dimensional Lorentzian vector space, we define a Witt basis $\{e_u,e_v,e_3\}$ as one which satisfies $\chi(e_u,e_v)=\chi(e_3,e_3)=1$ and $\chi(e_u,e_u)=\chi(e_v,e_v)=\chi(e_u,e_3)=\chi(e_v,e_3)=0$.} $v_1,v_2,v_3 \in \mathfrak{h}$  such that the Lie brackets are given by:
\begin{equation}\label{basis:eq}
[v_1,v_2]=0\, , \quad [v_1,v_3]=0\, , \quad [v_2,v_3]=-2 k v_1\,, \quad k \in \mathbb{R}^{>0}\,.
\end{equation}
Note that $k\neq 0$ cannot be absorbed into a redefinition of the basis $\{v_1,v_2, v_3\}$. (However, it can be assumed positive, since 
the sign can be always switched by multiplying the vectors of the basis by minus one.)
We observe that the center of the Heisenberg algebra, to which we will refer as the Heisenberg center, is spanned by $v_1$, which can be timelike, spacelike or lightlike.
Given a pseudo-Riemannian Heisenberg algebra $(\mathfrak{h},\chi)$, the isometry type of the corresponding 
pseudo-Riemannian metric on $\mathrm H$ is uniquely fixed after the specification of an orthonormal or Witt basis $\{v_1,v_2,v_3\}$ satisfying (\ref{basis:eq}).

 The three-dimensional Heisenberg group can be realized as $\mathbb{R}^3$ together with the following product:
\begin{equation}
(x,y,z) \cdot (a,b,c)=(a+x,b+y,c+z+ya-xb)\, , \quad (x,y,z), (a,b,c) \in \mathbb{R}^3\,.
\label{eq:r3realization}
\end{equation}
From \eqref{eq:r3realization} one can readily get   
a basis\footnote{Not necessarily orthonormal.} 
$\{w_1,w_2,w_3\}$ of $\mathfrak h$, given by the left-invariant vector fields
\begin{equation}
w_1=\p_z \,, \quad w_2=\p_x+k y \p_z \, ,\quad w_3=\p_y-k x \p_z\,,
\label{eq:coordh}
\end{equation}
where $(x,y,z)$ are standard coordinates on $\mathbb{R}^3$.
In particular, note that the only non-vanishing Lie bracket of these vectors is that of $[w_2,w_3]=-2 k w_1$. The dual basis of one-forms $\{w^1,w^2,w^3\}$ is given by:
\begin{equation}
w^1=\d z +k x \d y-k y\d x\,, \quad w^2=\d x \, ,\quad w^3=\d y \,. 
\label{eq:dualbasisheis}
\end{equation}

\subsection{Heisenberg four-manifolds}

\begin{mydef}
A four-dimensional pseudo-Riemannian manifold $(M,g)$ is said to be a \emph{Heisenberg four-manifold} if it is foliated by the orbits of a principal and isometric action of the three-dimensional Heisenberg group. 
\end{mydef}


\noindent 
 Note that a Heisenberg four-manifold $(M,g)$ admits an  $\mathrm{H}$-equivariant diffeomorphism identifying $M$ with $\mathcal{I}\times \mathrm{H}$, where  $\mathcal{I} \subset \mathbb{R}$ is either an open interval or a circle. Replacing $M$ by its universal covering, if necessary, we can assume the former.

Within the class of Heisenberg four-manifolds $(M,g)$, we shall restrict ourselves to those for which the restriction of $g$ to the leaves is non-degenerate. In such a case, the metric $g$ can be written in the form:
\begin{equation}\label{eq:methinv}g = \varepsilon dt^2 + \chi_t, \end{equation}
where $\varepsilon=\pm 1$  and $\{\chi_t\}_{t \in \mathcal{I}}$ is a family of left-invariant metrics on $\mathrm{H}$ parametrized by the time coordinate $t$. The different $\mathrm{H}$-orbits are identified by means of the normal geodesic flow\footnote{We have also properly reparametrized $t$ in order to correspond to the arc length parameter along a normal geodesic.} generated by $\partial_t$, such  that $\{(H,\chi_t)\}_{t \in \mathcal{I}}$  defines a family of pseudo-Riemannian Heisenberg groups. Up to orthogonal transformations, we can associate it to a family $\{ (\fre^t_i)\}_{t \in \mathcal{I}}$ of left-invariant 
sections $(\fre^t_i)$ of the frame bundle $\mathrm{F}(\mathrm{H})$ such that $(\fre^t_i) \in \mathrm{F}(\mathrm{H})$ is an orthonormal ($i=1,2,3$) or Witt ($i=u,v,3$) frame for $\chi_t$. Then 
$\{ \partial_t, (\fre_i^t)\}$ conforms an orthonormal or Witt frame for $(M,g)$. For ease of notation
we may denote the triplet $(\fre^t_i)$ simply by $\fre^t_i$, thus writing $\{ \fre^t_i\}_{t \in \mathcal{I}}$ and $\{ \partial_t, \fre_i^t\}$ instead of $\{ (\fre^t_i)\}_{t \in \mathcal{I}}$ and $\{ \partial_t, (\fre_i^t)\}$.
Analogously, we denote the corresponding family of dual orthonormal or Witt coframes on $\{(\mathrm{H},\chi_t)\}_{t \in \mathcal{I}}$ by $\{ \fre_t^i\}_{t \in \mathcal{I}}$ and the corresponding dual orthonormal or Witt coframe on $(M,g)$ by $\{ \d t, \fre_t^i\}$. 

\noindent
Having said this, we consider $g$ to have the form:
\begin{equation}
\label{eq:metgen}
g=\varepsilon \d t^2+\eta_{ij} \fre_t^i \otimes \fre_t^j\,, \qquad \varepsilon= \pm 1\,,
\end{equation}
\noindent 
where:
\begin{equation}
\label{eq:etaexpression}
\eta=\begin{pmatrix}
\varepsilon & 0 & 0 \\ 0 & 1 & 0 \\ 0& 0 & 1
\end{pmatrix} \quad  \mathrm{for\,\, orthonormal \,\, bases\,\, }\,,\quad   \eta=\begin{pmatrix}
0 & 1 & 0 \\ 1 & 0 & 0 \\ 0& 0 & 1
\end{pmatrix} \quad \mathrm{for\,\, Witt \,\, bases}\,.
\end{equation}

\noindent
Given the equivalent description of metrics \eqref{eq:methinv} in terms of families of orthonormal or Witt frames $\{\fre_i^t\}_{t \in \mathcal{I}}$ in $\mathrm{H}$, we may  think of metrics on $M$ as time-evolutions 
of frames on $\mathrm{H}$. (The evolution will be determined later from the self-dual Einstein equations.) In fact, let $t_0 \in \mathcal{I}$ be an initial time and  
$\fre_i^{t_0}$ an \emph{initial} orthonormal or Witt frame for the initial metric $\chi_{t_0}$. Then the time-evolution $\{\fre_i^t\}_{t \in \mathcal{I}}$ of such initial frame determines the metric $\{\chi_t\}_{t \in \mathcal{I}}$ and, in turn, the four-dimensional pseudo-Riemannian manifold $(M,g)$. We can write
\begin{equation}
\mathfrak{e}_i^t=U_{ij}^t \fre_j^{t_0}\, , \quad \mathfrak{e}_t^i=\fre^j_{t_0} (U^{t})^{-1}_{ji}\, , \quad 
\, U^t \in  \mathrm{GL}(3,\mathbb{R})\,, \quad U^{t_0}=\mathrm{Id}\,,
\end{equation}
\noindent
where $U^{t_0}=\mathrm{Id}$ is the initial condition for the time-evolution now encoded in $t\mapsto U^t$.

\noindent
We denote by $\mathcal{Z}\subset \mathfrak h= T_eH$ the  Heisenberg center and by 
$\mathcal{Z}_t\subset T_{(t,e)}(\{t\}\times H)\cong \mathfrak h$ the (constant) line which corresponds to $\mathcal Z$ under 
the canonical identification $\{t\}\times H\cong H$.


\begin{mydef}
Let $(M,g)$ a Heisenberg four-manifold of neutral signature. We say it is \emph{timelike, spacelike} or \emph{lightlike} if $\mathcal{Z}_t$ is timelike, spacelike or lightlike respectively for every $t \in \mathcal{I}$. 
\end{mydef}
It will be shown later that the causal character of $\mathcal{Z}_t$ is constant for Einstein Heisenberg four-manifolds.

\subsection{Choice of adapted frames for Einstein Heisenberg four-manifolds}
On studying Riemannian Heisenberg four-manifolds and neutral-signature timelike, spacelike or lightlike Heisenberg four-manifolds $(M,g)$, it is convenient to use the following special frames $\{ \fre_i^t\}_{t \in \mathcal{I}}$ to describe the four-dimensional metric $g$. 
\begin{prop}\label{prop:1stNF}Let $(M,g)$ be a Heisenberg four-manifold.
\begin{itemize}
\item If $\varepsilon=1$, then $(M,g)$ is Riemannian and there exists an orthonormal frame $\{ \fre_i^t\}_{t \in \mathcal{I}}$ such that $\fre_1^{t}$ generates $\mathcal{Z}_t$  and $\fre_2^{t}$ is a linear combination of $\fre_2^{t_0}$ and $\fre_3^{t_0}$ for all $t \in \mathcal{I}$.
This implies:
\begin{equation}
U^t=\begin{pmatrix}
a(t) & 0& 0\\
0 & b(t) & f(t) \\
j(t) & h(t) & c(t)\\ \end{pmatrix}\,, \quad [\fre_2^{t_0}, \fre_3^{t_0}]=-2 k \fre_1^{t_0}, \quad k >0\,,
\label{eq:canlnriem}
\end{equation}
where $a,b,c,f,j,h \in C^\infty(\mathcal{I})$.
\item If $\varepsilon=-1$, then $(M,g)$ is of neutral signature and:
\begin{itemize}
\item If $(M,g)$ is a timelike Heisenberg four-manifold,  then there exists an orthonormal frame $\{ \fre_i^t\}_{t \in \mathcal{I}}$ such that $\fre_1^{t}$ spans $\mathcal{Z}_t$ and $\fre_2^{t}$ is a linear combination of $\fre_2^{t_0}$ and $\fre_3^{t_0}$ for all $t \in \mathcal{I}$.
 The conclusion is again (\ref{eq:canlnriem}).
\item For spacelike Heisenberg four-manifolds $(M,g)$, we may choose\footnote{The reason to choose $\fre^t_3$ rather than $\fre_1^t$ is (\ref{eq:etaexpression}), where we fixed
the first vector of the orthonormal basis to be timelike.} the center $\mathcal{Z}_t$ to be spanned by $\fre_3^t$ for every $t \in \mathcal{I}$ and thus we may use the ansatz:
\begin{equation}
U^t=\begin{pmatrix}
c(t) & h(t)& j(t)\\
f(t) & b(t) & p(t) \\
0 &0 & a(t)\\ \end{pmatrix}\,, [\fre_1^{t_0}, \fre_2^{t_0}]=-2 k \fre_3^{t_0}\,, \quad k > 0 \,,
\label{eq:canlbspace}
\end{equation}
where 
$a,b,c,f,j,h,p \in C^\infty(\mathcal{I})$.

\item If $(M,g)$ is a lightlike Heisenberg four-manifold,  we define $\fre_u^{t}=\frac{1}{\sqrt{2}}(\fre_2^{t}-\fre_1^{t})$ and $\fre_v^{t}=\frac{1}{\sqrt{2}}(\fre_1^{t}+\fre_2^{t})$, where $(\fre_1^t,\fre_2^t,\fre_3^t)$ is an orthonormal frame such that $\fre_u^{t_0}\in \mathcal{Z}_{t_0}$. For a certain interval $\mathcal{I}'_l \subset \mathcal{I}$ containing $t_0$, we may choose $\fre_3^t$ to be parallel to $\fre_3^{t_0}$ for every $t \in \mathcal{I}'_l$. Then $\{\fre_u^t,\fre_v^t,\fre_3^t\}$ is a (local) Witt basis such that\footnote{Differently from the timelike and spacelike cases, where we impose a vector of the time-dependent orthonormal basis to be parallel to the Heisenberg center, here we opt to fix the spacelike direction in the time-evolving Witt basis $\{\fre_i^t\}_{t \in \mathcal{I}'_l}$, so that $\fre_3^t$ remains parallel to $\fre_3^{t_0}$ and $\fre_u^t$ freely changes. As shown in Proposition~\ref{prop:zcausalconst}, if we impose the lightlike Heisenberg four-manifold to be Einstein, it occurs additionally that $\fre_u^t$ stays parallel to the direction of the center given by $\fre_u^{t_0}$.}:
\begin{equation}
U_W^t=\begin{pmatrix}
c(t) & h(t)& j(t)\\
f(t) & b(t) & p(t) \\
0 &0 & a(t)\\ \end{pmatrix}\,, [\fre_v^{t_0}, \fre_3^{t_0}]=-2 k \fre_u^{t_0}\,, \quad k > 0 \,,
\label{eq:canlblight}
\end{equation}

where we write $U_W^t$ to indicate that this matrix is related to the Witt basis and $a,b,c,f,j,h,p \in C^\infty(\mathcal{I}'_l)$. 
\end{itemize}
\end{itemize}
\end{prop}

\begin{proof}

In the case $\varepsilon =1$ and in the case $\varepsilon = -1$ with timelike center the 
plane $E=\mathrm{span}\{ \fre_2^{t_0},\fre_3^{t_0}\}$ is spacelike and the line $\mathcal{Z}_t$ is definite with respect to $\chi_t$. 
Therefore, it suffices to choose a unit vector 
$\fre_1^t\in \mathcal{Z}_t$, a unit vector $\fre_2^t\in  E$ perpendicular to $\fre_1^t$ with respect to $\chi_t$
and to complement this pair to an orthonormal frame.   When $\varepsilon = -1$ with spacelike center, 
we can still choose a unit vector $\fre_3^t\in \mathcal{Z}_t$, which suffices for the claim. Note that we cannot specialize further our choice of ansatz, since we do not know \emph{a priori} the space-time character of the intersection of $\mathrm{span}\{\fre_1^{t_0},\fre_2^{t_0}\}$ with $\mathcal{Z}_t^{\perp_t}$, the orthogonal complement to the center with respect to $\chi_t$. Finally, in the lightlike case there exists a subinterval $\mathcal{I}'_l \subset \mathcal{I}$ containing $t_0$ in which the line generated by $\fre_3^{t_0}$ is spacelike, since this is an open condition. Choosing $\fre_3^t$ to be parallel to $\fre_3^{t_0}$ for every $t\in \mathcal{I}'_l$ we conclude.  
\end{proof}



\begin{prop}
\label{prop:zcausalconst}
Let $(M,g)=(\mathcal{I} \times \mathrm{H}, \varepsilon \d t^2+\chi_t)$ be an Einstein Heisenberg four-manifold of neutral signature. Assume that $\mathcal{Z}_{t_0}$ for some $t_0 \in \mathcal{I}$ is timelike (resp. spacelike or lightlike). Then $\mathcal{Z}_t$ remains  timelike (resp. spacelike or lightlike) for all $t \in \mathcal{I}$. 
\end{prop}
\begin{proof}
First we will show that if the causal or space-time character of $\mathcal{Z}_{t_0}$ is lightlike, then it remains invariant in an open subinterval of $\mathcal{I}$ containing $t_0$. For that, assuming that $\mathcal{Z}_{t_0}$ is lightlike,  in an open subinterval $\mathcal{I}'_l$ we pick up the ansatz \eqref{eq:canlblight} for $U^t$ and the Lie brackets at $t_0$. If $\ric^g$ denotes the Riemann tensor of $(M,g)$, we have:
\begin{eqnarray*}
\ric^g(\partial_t,\fre_v^t)&=&k\frac{-2a b(b c -f h) (bj-hp)a'+f' (bj-hp)^3 }{a^2(bc-fh)^2}\\ &+&k\frac{(bj-hp)^2(c(p b'-b p'))+f(-j b'+h p'))}{a^2(bc-fh)^2}\\ &+&k \frac{ -2 f h^2 p b'+ b^3 c j'-b^2(h(2 j f'+fj'))+c(-j b'+p h'+h p') }{(bc-fh)^2}\\ & &+ k \frac{   bh(2hp f'+f(jb'+ph'+hp'))}{(bc-fh)^2}\,,
\end{eqnarray*}
\begin{eqnarray*}
\ric^g(\partial_t,\fre_3^t)&=&-k\frac{b^3 j (j c'-c j')+b^2( c h p j'+2 a^2(-h c'+ch')+j^2(h f'-f h'))}{a^2 (bc-fh)^2}\\&-&k\frac{b^2 j( c p h'+h(-2pc'+f j'-cp')+h^2(2a^2 f b'+p^2(-cb'+h f'))}{a^2 (bc-fh)^2}\\ &-&k \frac{p f h^2(j b'-h p')+bh^2(p^2c'-2a^2 f'-2 j p f')}{a^2 (bc-fh)^2}\\&-&k \frac{b h c p(j b'-p h'+ hp')+bh f(-j^2 b'-hp j'+j(ph'+hp'))}{a^2 (bc-fh)^2}\,.
\end{eqnarray*}
The Einstein condition implies that $\ric^g(\partial_t, \fre_v^t)=\ric^g(\partial_t, \fre_3^t)=0$. We can solve for $j'(t)$ and $h'(t)$ and we get:
\begin{eqnarray*}
j'(t)&=&\frac{q_1(a,b,c,f,h,j,p,a',b',c',f',p')}{a^2 b(bc-fh) (2a^2 b^2c-f(bj-hp)^2)}\, , \\ h'(t)&=&\frac{q_2(a,b,c,f,h,j,p,a',b',c',f',p')}{a^2 b (2a^2 b^2c-f(bj-hp)^2)}\, ,
\end{eqnarray*}
where $q_1,q_2: \mathbb{R}^{12} \rightarrow \mathbb{R}$ are polynomials that vanish on the subspace $V \subset \mathbb{R}^{12}$ defined by $j=h=0$. (Note that the denominators are non-zero in a neighborhood of $V$, by non-degeneracy of $U^t_W$.) 
By a continuity argument, we can guarantee that $h(t)=j(t)=0$ in $\mathcal{I}'_l$ and therefore $\mathcal{Z}_t$ is lightlike for all $t \in \mathcal{I}'_l$.

\noindent
 Similarly, if $\mathcal{Z}_{t_0}$ is timelike (resp. spacelike), then $\mathcal{Z}_{t}$ remains timelike (resp. spacelike) in a local subset of $\mathcal{I}$ around $t_0$, since these are open conditions. Therefore, the three subsets $\mathcal{I}_{\mathrm{time}}$,   $\mathcal{I}_{\mathrm{space}}$,  $\mathcal{I}_{\mathrm{light}}$ of $\mathcal I$ consisting
of $t$ at which $\mathcal{Z}_t$ is  respectively timelike, spacelike or lightlike are open. By connectedness of $\mathcal I$ this proves that $\mathcal{I}$ coincides 
with one of these three subsets.
\end{proof}
\begin{rem}
When studying neutral-signature Einstein Heisenberg four-manifolds, we see that the types of timelike, spacelike and lightlike Einstein Heisenberg four-manifolds exhaust all possibilities, since Proposition \ref{prop:zcausalconst} does not allow changes of the causal type of $\mathcal{Z}_t$. 
\end{rem}
\begin{rem}
Observe that the Einstein condition  in the previous proposition is too strong, and one may actually require much weaker conditions for the proposition to equally hold, such as requiring that $\ric^g(\partial_t) \vert_{\{t \times H\}}=0$. 
\end{rem}
\begin{prop}\label{prop:2ndNF}
Let $(M,g)$ be a Riemannian or neutral-signature Einstein Heisenberg four-manifold. Then:
\begin{itemize}
\item If $(M,g)$ is either Riemannian or  of neutral-signature and timelike, then we can choose:
\begin{equation} 
U^t=\begin{pmatrix}
a(t) & 0& 0\\
0 & b(t) & 0 \\
0 & h(t) & c(t)\\ \end{pmatrix}\,, \quad [\fre_2^{t_0}, \fre_3^{t_0}]=-2 k \fre_1^{t_0}, \quad k >0\,.
\label{eq:ansriemsim}
\end{equation}
\item If $(M,g)$ is neutral-signature and spacelike, we can choose the following ansatz in an open subinterval $\mathcal{I}'_s \subset \mathcal{I}$ containing $t_0$:
\begin{equation}
U^t=\begin{pmatrix}
c(t) &h(t) & 0\\
-h(t) & b(t) & 0 \\
0 & 0 & a(t) \\
\end{pmatrix}\,, \quad [\fre_1^{t_0}, \fre_2^{t_0}]=-2 k \fre_3^{t_0}, \quad k >0\,.
\label{eq:ansspacesim}
\end{equation}
\item If $(M,g)$ is neutral-signature and lightlike, we can pick up in an open interval $\mathcal{I}'_l \subset \mathcal{I}$:
\begin{equation}
U_W^t=\begin{pmatrix}
1 & 0 & 0 \\
f(t) & b(t) & p(t) \\
0 &0 & a(t)\\ \end{pmatrix}\,, [\fre_v^{t_0}, \fre_3^{t_0}]=-2 k \fre_u^{t_0}\,, \quad k > 0 \,.
\label{eq:anslightsim}
\end{equation}
\end{itemize}
\end{prop}
\begin{proof}
Regarding the neutral-signature case:
\begin{itemize}
\item If $(M,g)$ is timelike, we can choose for every $t \in \mathcal{I}$ the ansatz \eqref{eq:canlnriem} for the matrix $U^t$ and for the Lie brackets at $t_0$. If $\ric^g$ denotes the Ricci tensor of $(M,g)$, we find:
\begin{equation}
\ric^g(\partial_t, \fre_2^{t})=-\frac{k \left(b \left(c j'-j c'\right)+f \left(j h'-h j'\right)\right)}{a^2}\, , \quad \ric^g(\partial_t, \fre_3^{t})=-\frac{k   j \left(b f'-f b'\right)}{a^2}\,.
\label{eq:algunosriccitime}
\end{equation}
If $(M,g)$ is in addition Einstein, then the previous components must identically vanish. Assume $j(t) \neq 0$. Then we would find that $f(t)= f_0 b(t)$ in an interval $\mathcal{I}_0 \subset \mathcal{I}$ in which $b(t) \neq 0$ (which always exists given that $b(t_0)=1$). Since $f(t_0)=0$, then $f(t)=0$ in $\mathcal{I}_0$. Assume now there exists a lower (or upper) bound $t_1\in \mathcal{I}$ for $\mathcal{I}_0$  such that $b(t_1)=0$. By continuity we would have that $f(t_1)=0$ and the matrix $U^t$ would be degenerate at $t_1$, what contradicts the fact that $\{\fre_i^t\}_{t \in \mathcal{I}}$ is an orthonormal basis. Then we learn that such $t_1$ does not exist and $f(t)=0$ in the entire interval in $\mathcal{I}$, which in turn implies that $j(t)=j_0 c(t)$ in some open subinterval of $\mathcal{I}$. Owing the fact that $j(t_0)=0$, by the same reasoning as above we conclude that $j(t)=0$ in the whole $\mathcal{I}$.

\item If $(M,g)$ is spacelike we can choose for every $t \in \mathcal{I}$ the ansatz \eqref{eq:canlbspace} for $U^t$ and the Lie brackets at $t_0$. If $\ric^g$ denotes the Ricci tensor of $(M,g)$, we have:
\begin{eqnarray*}
\begin{split}
\ric^g (\partial_t, \fre_1^t)&=&k\frac{f'(b j-h p) +c(p b'-b p') +f(h p'-j b'))}{a^2} \,, \\
\ric^g( \partial_t, \fre_2^t)&=&k\frac{b(j c'-c j') +h(f j'-p c') +h'(c p-f j))}{a^2}\,.
\end{split}
\end{eqnarray*}
In order to have $\ric^g(\partial_t, \fre_1^t)=\ric^g(\partial_t, \fre_2^t)=0$ we must demand:
\begin{equation*}
\begin{split}
p'(t)&=\frac{c p b'-f j b'+f' (b j-h p)}{b c-f h}\,,\\
j'(t) &=\frac{b j c'-h p c'+h' (c p-f j)}{b c-f h}\,.\\
\end{split}
\end{equation*}
However, taking into account the initial conditions $a(t_0)=b(t_0)=c(t_0)=1$,  $f(t_0)=h(t_0)=j(t_0)=p(t_0)=0$ and $bc-fh \neq 0$ (from the non-degeneracy of $U^t$), through the use of the uniqueness and existence theorem of ODEs we infer that $p(t)=j(t)=0$ for all $t \in \mathcal{I}$. Performing now appropriate $\mathrm{SO}(1,1)$ rotations, we can finally impose $U_t$ to have the same form as in \eqref{eq:ansspacesim} in an open subinterval $\mathcal{I}'_s \subset \mathcal{I}$.

\item If $(M,g)$ is lightlike, then \eqref{eq:anslightsim} just follows from the results obtained in the proof of Proposition \ref{prop:zcausalconst} and by absorbing\footnote{Another way to see this is by noticing that $\fre_t^u=\frac{1}{c(t)} \fre_{t_0}^u$, so that $g=- \d t^2+ \frac{1}{c(t)}\fre_{t_0}^u \odot \fre_t^v+ \fre_t^3  \otimes \fre_t^3$. By redefining $\frac{1}{c(t)} \fre_{t}^v \rightarrow \fre_t^v$, we observe that we can set $c(t)=1$. } the factor $c(t)$ in the functions $f(t), b(t)$ and $p(t)$, which yields the same metric.
\end{itemize}
Regarding the Riemannian case, using \eqref{eq:canlnriem} we compute:
\begin{equation}
\ric^g(\partial_t, \fre_2^{t})=\frac{k  \left(b \left(c j'-j c'\right)+f \left(j h'-h j'\right)\right)}{a^2}\, , \quad \ric^g(\partial_t, \fre_3^{t})=\frac{k   j \left(b f'-f b'\right)}{a^2}\,.
\end{equation}
We observe that, up to a global sign, this is exactly the same as the result obtained in \eqref{eq:algunosriccitime} for timelike Einstein Heisenberg four-manifolds. Then we equivalently conclude that $j(t)=0$, and by an appropriate $\mathrm{SO}(2)$ rotation, we arrive to \eqref{eq:ansriemsim}.
\end{proof}
\section{Quaternionic (para)K\"ahler Heisenberg four-manifolds}

\label{sec:3}

In this section we classify all Heisenberg four-manifolds which satisfy the condition of being quaternionic (para)K\"ahler. For that, we revise first the definition of a quaternionic (para)K\"ahler four-manifold. Recall first that an orientable pseudo-Riemannian four-manifold of Riemannian or neutral signature is called \emph{half-conformally flat} if 
its Weyl tensor is self-dual for one of the two orientations.
\begin{mydef}
\label{def:qk}
Let $(M,g)$ be a Riemannian or neutral-signature orientable four-manifold. It is said to be a \emph{quaternionic K\"ahler} (resp.\ \emph{quaternionic paraK\"ahler}) if and only if it is Einstein with non-zero Einstein constant and half-conformally flat. We shall refer to them jointly as \emph{quaternionic (para)K\"ahler} four-manifolds. 
\end{mydef}
\begin{rem}We observe that the definition of a quaternionic (para)K\"ahler four-manifold cannot be applied for larger dimensions, since the notion of self-duality is restricted to four-dimensions. Actually, for dimensions $n=4m$ with $m>1$ the definition of quaternionic (para)K\"ahler manifolds is that of pseudo-Riemannian manifolds admitting a parallel 
skew-symmetric (para)quaternionic structure $Q$. We recall that such a structure $Q$ is locally spanned by three anti-commuting endomorphism fields $I, J, K=IJ$ such that 
$I^2=J^2 = \pm \mathbf{1}$. However, in four dimensions this definition is too weak, since every orientable four-manifold satisfies it, and it turns out that the natural definition of quaternionic (para)K\"ahler four-manifold is that of Definition \ref{def:qk} \cite{B}.
\end{rem}
Let $\mathrm{W}^g$ denote the Weyl tensor of $(M,g)$. We define the \emph{Weyl self-duality tensor} $\mathcal{W}^g$ as the $(0,4)$-tensor given by:
\begin{equation}
\mathcal{W}^g(X,Y,U,V)= g((\star \mathrm{W}^g)(X,Y)U,V)-g(\mathrm{W}^g(X,Y)U,V)\, , \quad X,Y,U,V \in \mathfrak{X}(M)\,,
\label{eq:eqhpp}
\end{equation}
\noindent
where $\star$ denotes the Hodge star map with respect to a given orientation on $(M,g)$. 
Up to a factor, it is the antiself-dual part of the Weyl tensor and, hence, the obstruction to $(M,g,\star )$ being self-dual.
 Then we have that a four-manifold will be quaternionic (para)K\"ahler if and only if
\begin{equation}
\mathcal{W}^g=0\, , \quad \mathrm{Ric}^g=\Lambda g \, , \quad \Lambda \in \mathbb{R}\setminus \{0\}\,,
\label{eq:qkcond}
\end{equation}
for one of the two orientations, where $\mathrm{Ric}^g$ denotes the Ricci tensor of $(M,g)$. (We will 
always consider the orientation such that $\mathcal{W}^g=0$.)

\noindent 
Now we proceed to the classification of quaternionic (para)K\"ahler Heisenberg four-manifolds attending to the signature of the metric and the space-time character of $\mathcal{Z}_{t_0}$.

\subsection{Quaternionic K\"ahler and timelike quaternionic paraK\"ahler Heisenberg four-manifolds}

We start by classifying all (Riemannian) quaternionic K\"ahler Heisenberg four-manifolds and all timelike quaternionic paraK\"ahler Heisenberg four-manifolds. Since we will carry out such classification simultaneously, it is convenient to coin the term (timelike) quaternionic (para)K\"ahler Heisenberg four-manifolds to refer to both of them at once. To particularize our results to one of these cases, we just have to set $\varepsilon=\pm 1$ correspondingly. 

\noindent
The reason for treating them at the same time is that we can use the same identical ansatz to describe the three-dimensional metric $\chi_t$ for both quaternionic K\"ahler and timelike quaternionic paraK\"ahler Heisenberg four-manifolds. Indeed, such ansatz is given by \eqref{eq:ansriemsim}:
\begin{equation} 
U^t=\begin{pmatrix}
a(t) & 0& 0\\
0 & b(t) & 0 \\
0 & h(t) & c(t)\\ \end{pmatrix}\,, \quad [\fre_2^{t_0}, \fre_3^{t_0}]=-2 k \fre_1^{t_0}, \quad k >0\,.
\label{eq:ansriemtsim}
\end{equation}
\begin{prop}
\label{prop:riccitime}
The non-zero components of the Ricci  tensor $\mathrm{Ric}^g$  of the metric obtained from \eqref{eq:ansriemtsim} are:
\begin{eqnarray*}
\mathrm{Ric}^g(\partial_t ,\partial_t)&=&-\frac{2 \left(a'\right)^2}{a^2}+\frac{a''}{a}-\frac{\left(4 b^2+h^2\right) \left(c'\right)^2+c^2 \left(4 \left(b'\right)^2+\left(h'\right)^2-2 b
   b''\right)}{2 b^2 c^2}\\ & & +\frac{\left(h c' h'+b^2 c''\right)}{b^2 c}\, ,\\
\mathrm{Ric}^g(\mathfrak{e}_1^t,\mathfrak{e}_1^t)&=&\frac{2 b^3 c^3 k^2-2 b c \left(a'\right)^2-a a' \left(c b'+b c'\right)+a b c a''}{a^2 b c} \, ,\\
\mathrm{Ric}^g(\mathfrak{e}_2^t,\mathfrak{e}_2^t)&=&-\frac{4 b^4 c^4 k^2+a^2 \left(4 c^2 \left(b'\right)^2-\left(h c'-c h'\right)^2\right)+2 a b c \left(c a' b'+a b' c'-a c b''\right)}{2
   a^2 b^2 c^2 \varepsilon }\, ,\\
   \mathrm{Ric}^g(\mathfrak{e}_2^t,\mathfrak{e}_3^t) &=& \frac{ c \left(h c'-c h'\right)(3ab'+ba')+a b \left(h \left(c'\right)^2-c c' h'+c \left(-h c''+c h''\right)\right)}{2 a
   b^2 c^2 \varepsilon }\, , \\
\mathrm{Ric}^g(\mathfrak{e}_3^t,\mathfrak{e}_3^t)&=&-\frac{4 b^4 c^4 k^2+2 a^2 b c b' c'+a^2 \left(h c'-c h'\right)^2+2 a b^2 \left(c a' c'+2 a \left(c'\right)^2-a c c''\right)}{2 a^2 b^2
   c^2 \varepsilon }\, .\\
\end{eqnarray*}
\end{prop}
\begin{proof}
Just by direct computation.
\end{proof}
\begin{prop}
\label{prop:edosqk}
Let $(M,g)$ be a (timelike) quaternionic (para)K\"ahler Heisenberg four-manifold. Then:
\begin{equation}
\label{eq:edosqk}
a'(t)=-\frac{k^2 b^6 +a^2(\varepsilon \Lambda b^2+(b')^2)}{2ab b'},\quad \varepsilon \Lambda = \frac{3 (k b^3- ab')^3}{ a^2 b^2(-3k b^3+a b')}\,, \quad c=b \, , \quad  h=0\,.
\end{equation}
\end{prop}
\begin{proof}
A (timelike) quaternionic (para)K\"ahler Heisenberg four-manifold is Einstein (with non-zero Einstein constant) and is half-conformally flat.

\noindent 
 Firstly, we observe that the Einstein condition imposes that $\ric^g(\fre_2^t, \fre_3^t)=0$. By Proposition \ref{prop:riccitime}, one can solve for $h''$ and get:
\begin{equation}
h''=\frac{-h (c')^2+c c' h'+c h c''}{c^2}+\frac{c a' (-h c'+c h')}{ac^2}+\frac{3c b' (-h c'+c h')}{bc^2}\,.
\end{equation}
Now we move into the Weyl self-duality tensor $\mathcal{W}^g$ defined back at \eqref{eq:eqhpp}. The component $\mathcal{W}^g(\fre_1^t,\fre_2^t,\fre_1^t,\fre_3^t)$ reads:
\begin{equation}
\mathcal{W}^g(\fre_1^t,\fre_2^t,\fre_1^t,\fre_3^t)=-\frac{(k bc+a')(-h c'+ch')}{2 a bc}\,.
\end{equation}
For the latter to be zero, either $k bc+a'=0$ or $-h c'+ch'=0$. If $k bc+a'=0$, we would find that the Einstein constant has to vanish, so if we assume that $\Lambda \neq 0$, we must have $-h c'+ch'=0$, which implies in turn that $h(t)=h_0 c(t)$. However, since $h(t_0)=0$, then $h(t)=0$. 

\noindent 
Setting $h=0$ in Proposition \ref{prop:riccitime}, on imposing the Einstein condition $\ric^g=\Lambda g$ we can solve for $a''(t), b''(t), c''(t)$ and $a'(t)$ and obtain:
\begin{eqnarray}
\label{eq:firstariem}
a'(t)&=&-\frac{k^2 b^3 c^3+a^2(\varepsilon \Lambda b c+b' c')}{a(cb'+b c')}\, , \\
\nonumber
a''(t)&=&\frac{2k^4 b^7 c^7 +k^2 a^2 b^3 c^3(-3c^2(b')^2-2bc b' c'+b^2(4 \varepsilon \Lambda c^2-3 (c')^2)}{a^3 b c(cb'+c' b)^2}\\ \label{eq:secondariem}& & +\frac{a^4 (2 \Lambda^2 b^3 c^3-c^2 (b')^3 c'-b^2 b' c'(-4 \varepsilon \Lambda c^2+(c')^2))}{a^3 b c(cb'+c' b)^2}\,,\\
\label{eq:secondbriem}
b''(t) &=& \frac{k^2 c^3 b^4(2bc'+cb')+a^2(b^2 (c')^2b'+2c^2(b')^3+bcc'(\varepsilon\Lambda b^2+2(b')^2))}{a^2 b c (cb'+bc')} \,,\\
\label{eq:secondcriem}
c''(t) &=&\frac{k^2 b^3 c^4(2cb'+bc')+a^2(c^2 (b')^2c'+2b^2(c')^3+bcb'(\varepsilon\Lambda c^2+2(c')^2))}{a^2 b c (cb'+bc')} \,.
\end{eqnarray}
More precisely, these equations should be written with the $(cb'+bc')$-factors on the left-hand side, to avoid possible zeros of the denominator, 
which we will not do to keep the formulas simple. The same comments apply to the formulas below.
On substituting these results into the Weyl self-duality tensor, we encounter:
\begin{eqnarray*}
\mathcal{W}^g(\fre_1^{t} ,\fre_2^{t} ,\fre_1^{t} ,\fre_2^{t} )&=&\frac{-3k^3 b^5 c^4+3k^2 a b^3c^2(2cb'+bc')-3ka^2bc(\varepsilon \Lambda b^2 c+c (b')^2+2bb'c')}{3a^3 b(bc'+cb')}\\& &+ \frac{2 \varepsilon \Lambda b c b'-\varepsilon \Lambda b^2 c'+3 (b')^2 c'}{3  b(bc'+cb')}\,.
\end{eqnarray*}
From here one can solve for $\Lambda$ and obtain: 
\begin{eqnarray}
\Lambda=3 \frac{(kb^2c-ab')^2(-k bc^2+a c')}{ \varepsilon a^2 b(3k b^2 c^2-2 a c b'+ a bc')}\,.
\label{eq:lambdariem}
\end{eqnarray}
Taking this result into the rest of the components of $\mathcal{W}^g$, we find in particular:
\begin{eqnarray*}
\mathcal{W}^g(\p_t ,\fre_1^{t} ,\partial_t ,\fre_1^{t} )=\frac{\varepsilon (c b'-b c')(-k b^2 c+a b')(-k c^2 b +a c')}{a bc (3 k b^2 c^2-2  ac b'+ a b c')}\,.
\end{eqnarray*}
It follows that the first term, the second or the third term in brackets vanishes. If the second  or third one is identically zero, then we find that $\Lambda=0$, upon substitution in \eqref{eq:lambdariem}. Since we are assuming non-zero Einstein constant, we discard this possibility and then $c' b=b' c$, which in turn implies that $b=c$ since $b(t_0)=c(t_0)=1$. Imposing this condition, we find that all components of $\mathcal{W}^g$ vanish identically. Finally, we observe that \eqref{eq:lambdariem} can be simplified to take the form:
\begin{equation}
\Lambda=\frac{3 (k b^3-ab')^3}{\varepsilon a^2 b^2(-3k b^3+a b')}\,.
\end{equation}
Upon use of this expression, its first time derivative as well as equation \eqref{eq:firstariem} we check that \eqref{eq:secondariem},  \eqref{eq:secondbriem} and  \eqref{eq:secondcriem} are satisfied as well and we conclude.
\end{proof}

\begin{prop}\label{Weyl:prop}
Let $(M,g)$ be a (timelike) quaternionic (para)K\"ahler Heisenberg four-manifold. Then:
\begin{itemize}
\item The eigenvalues of the Weyl tensor, understood as a symmetric endomorphism of the bundle of two-forms, are given by $(-2\nu, \nu , \nu)$, where:
\begin{equation*}
\nu=\frac{16 k^3 b^7}{\varepsilon a^3(-3k b^3+a b')}\,.
\end{equation*}
\item $(M,g)$ is conformally K\"ahler for two complex structures with opposite orientations.
\end{itemize}
\end{prop}
\begin{proof}
Define the following triplet of self-dual two-forms:
\begin{equation}\label{SD:eq}
\omega_i= \d t \wedge \fre_t^i+ \star (\d t \wedge \fre_t^i)\, ,
\end{equation}
where $\star$ denotes the Hodge dual operation. Interpreting the Weyl tensor as a symmetric endomorphism of the bundle of two-forms in the canonical way\footnote{In components, $(\mathrm{W}^g(\omega_i))_{\mu \nu}=(\mathrm{W}^g)_{\mu \nu \rho \sigma} \omega_i^{\rho \sigma}$.}, we observe that:
\begin{equation*}
\mathrm{W}^g(\omega_1)=-2 \nu \omega_1 \, , \quad \mathrm{W}^g(\omega_2)= \nu \omega_2\, ,\quad \mathrm{W}^g(\omega_3)= \nu \omega_3\,.
\end{equation*}
This proves the first part of the proposition. Regarding the second one, we first note that the rescaled two-form
\begin{equation*}
\tilde{\omega}_1=\vert \mathrm{W}^g \vert_g^{2/3} \omega_1
\end{equation*}
is closed and satisfies that $\vert \tilde{\omega}_1 \vert_{\tilde{g}}^2=4$ with respect to the rescaled metric $\tilde{g}=\vert \mathrm{W}^g \vert_g^{2/3} g$, in agreement with 
\cite{D,AG} in the Riemannian case. We claim that $\tilde{g}$ is pseudo-K\"ahler with the K\"ahler form $\tilde{\omega}_1$. To prove the integrability of the almost complex structure $J_1 = \tilde{g}^{-1}\tilde{\omega}_1=g^{-1}\omega_1$ we use that the following rescaled two-forms 
\begin{equation*}
\tilde{\omega}_2=a b \, \omega_2\, , \quad \tilde{\omega}_3=a b \, \omega_3\,
\end{equation*}
are closed.  This is a consequence of Lemma~\ref{int:lem}, a simple generalization of the Hitchin lemma \cite{H}. By introducing a relative sign in (\ref{SD:eq}) 
we can likewise obtain a conformally K\"ahler structure $(M,g,J_1')$ for the opposite orientation. (So contrary to $J_1$ the complex structure $J_1'$ is not subordinate to the 
(para)quaternionic structure $Q = \mathrm{span} \{ J_1, J_2, J_3\}$.)
\end{proof}
\begin{lem} \label{int:lem}Let $(M,g)$ be a pseudo-Riemannian manifold endowed with two anti-commuting skew-symmetric endomorphism fields 
$J_2$, $J_3$ such that $J_2^2=J_3^2=\varepsilon \mathbf{1}$, where $\varepsilon = \pm 1$.  Assume that the two-forms $\omega_i = g\circ J_i $ are closed 
for $i=2,3$. Then $J_1 := J_2J_3$ is an integrable skew-symmetric complex structure. 
\end{lem}
\begin{proof} It is trivial to check that $J_1$ is a skew-symmetric almost complex structure. To prove the integrability 
we use that $\Omega = \omega_2 - i\varepsilon  \omega_3$ is of type $(2,0)$ with respect to $J_1$ and non-degenerate (as a complex bilinear form). 
Due to these properties, it suffices to show that $\Omega ([\bar{X}, \bar{Y}], Z) =0$, for all vector fields $X, Y, Z$ of type $(1,0)$, since this implies 
the involutivity of the $(-i)$-eigendistribution of $J_1$. 
This is an immediate consequence of $d\Omega=0$, since $(d\Omega)(\bar{X},\bar{Y},Z) = -\Omega ([\bar{X},\bar{Y}],Z)$.  
\end{proof}

\begin{rem}We have shown in the proof of Proposition~\ref{Weyl:prop} that $(M,g)$ is not only conformally K\"ahler but admits two almost (para)K\"ahler 
structures $\tilde{\omega}_2$ and $\tilde{\omega}_3$ compatible with a second conformally rescaled metric $g' := a b \, g$, such that 
$\vert \tilde{\omega}_2 \vert_{g'}^2=\vert \tilde{\omega}_3 \vert_{g'}^2=4 \varepsilon$. 
\end{rem}

\begin{rem}\label{remarktolemma}Note that Lemma \ref{int:lem} can be easily adapted to include the case $J_2^2=-J_3^2 = \varepsilon \mathbf{1}$, $\varepsilon= \pm 1$. The conclusion 
is then that $J_1$ is an integrable skew-symmetric paracomplex structure. In fact, in that case one can consider $\Omega = \omega_2 + e \varepsilon  \omega_3$, which takes values in the ring $\mathbb{R}[e]\cong \mathbb{R} \oplus \mathbb{R}$ generated by $e$ with the relation $e^2=1$ (the ring of paracomplex numbers).  Note that $\Omega$ has type $(2,0)$ in the sense that 
$\Omega (J_1 \cdot , \cdot ) = \Omega (\cdot , J_1 \cdot ) = 
e \Omega$ and is non-degenerate in the sense that a real vector $X$ satisfies $\Omega (X+eJ_1X,Y+eJ_1Y)=0$ for all 
real vectors $Y$ if and only if $X=0$. 
\end{rem}

The second equation in \eqref{eq:edosqk} is a cubic equation for $b'(t)$, and depending on the values of $k$ and $\Lambda$, we may have one or more real solutions. Define\footnote{Given a complex number $z \in \mathbb{C}$, it has always three cubic roots. When we write $z^{1/3}$ we will mean by convention the cubic root $z^{1/3}=\vert z \vert^{1/3} e^{i/3\,  \mathrm{arg}(z)}$.}:
\begin{eqnarray}
\nonumber
\mathcal{B}_l&=&\frac{1}{3a^3}\left ( 3  k a^2 b^3-\frac{  \varepsilon \Lambda a^6 b^2 e^{-2i \pi (l-1)/3}}{(9  k \varepsilon \Lambda a^8 b^5 + \sqrt{\Lambda^2 a^{16} b^6(\varepsilon \Lambda a^2+81 k^2 b^4) })^{1/3}}\right. \\& & \label{eq:desbriem1} +e^{2i \pi (l-1)/3} (9  k \varepsilon \Lambda a^8 b^5 + \sqrt{\Lambda^2 a^{16} b^6(\varepsilon \Lambda a^2+81 k^2 b^4) })^{1/3}\Bigg )\, , \quad l=1,2,3\,.
\end{eqnarray}
\begin{prop}
\label{prop:solsbrt}
Let $(M,g)$ be a (timelike) quaternionic (para)K\"ahler Heisenberg four-manifold. Then $a'=-\frac{k^2 b^6 +a^2(\varepsilon \Lambda b^2+(b')^2)}{2ab b'}$ and:
\begin{itemize}
\item If $\varepsilon\Lambda >-81 k^2$,  $b'=\mathcal{B}_1$.
\item If $\varepsilon\Lambda \leq -81 k^2$, there are three solutions for $b'$ obtained by setting $l=1,2,3$ in \eqref{eq:desbriem1}, $b'_l=\mathcal{B}_l$.
\end{itemize}
\end{prop}
\begin{proof}
The result for $a'$ was derived in Proposition \ref{prop:edosqk}. If we define now $\beta=k b^3- a b'$, the second equation in \eqref{eq:edosqk} is equivalent:
\begin{equation}
\beta^3+\frac{\varepsilon a^2 b^2 \Lambda}{3} \beta+ \frac{2k \varepsilon}{3} a^2 b^5 \Lambda=0\,.
\end{equation}
At $t_0$, this equation reads:
\begin{equation}
\beta^3+\frac{\varepsilon\Lambda}{3} \beta+ \frac{2k \varepsilon}{3} \Lambda=0\,.
\label{eq:cubicpolrt}
\end{equation}
The discriminant $\Delta$ of this equation takes the form:
\begin{equation}
\Delta=-\frac{4}{27} \Lambda^2  \left (\varepsilon \Lambda +81 k^2  \right) \,.
\end{equation}
By standard theory of cubic equations, if $\Delta<0$ then there is just one real solution to \eqref{eq:cubicpolrt} and if $\Delta \geq 0$ there exist three (maybe multiple) real roots. Thus if $\varepsilon \Lambda >-81 k^2$, there is only a unique real solution to \eqref{eq:cubicpolrt} (given by\footnote{It can be checked that  the expression \eqref{eq:desbriem1} with $l=1$ takes real values for all $\varepsilon\Lambda \neq 0$.} \eqref{eq:desbriem1} with $l=1$) and therefore there is a unique solution for $b'(t_0)$, which in turn produces one real solution for $(a(t), b(t))$ defined on an appropriate interval $\mathcal{I}$. If $\varepsilon \Lambda  \leq -81 k^2$, there are three real roots (with at least two identical roots when the equality holds) to \eqref{eq:cubicpolrt} and therefore there are three real solutions for $b'(t_0)$ (given by \eqref{eq:desbriem1}, $l=1,2,3$). These yield three real solutions for $(a(t), b(t))$, each defined on intervals $\mathcal{I}_l$.
\end{proof}
Proposition \ref{prop:solsbrt} provides a system of ordinary differential equations for $(a(t),b(t))$ with the initial condition $a(t_0)=b(t_0)=1$. By virtue of the theorem of existence and uniqueness of ordinary differential equations, it is enough to find one solution for each possible value of $k$ and $\varepsilon\Lambda$, since it will be unique\footnote{It can be seen that the conditions for the existence and uniqueness theorem to hold are satisfied, at least locally around the initial condition. Note that the derivatives are already solved at one side of the equations, which facilitates this check.}. In fact, we find it is convenient to split our study into three different possibilities according to the value of $\varepsilon \Lambda$, and we will distinguish between stationary, negative and positive (timelike) quaternionic (para)K\"ahler Heisenberg four-manifolds (to be defined below).

\subsubsection{Stationary (timelike) quaternionic (para)K\"ahler Heisenberg four-manifolds}

\begin{mydef}
\label{def:stationary}
A (timelike)  quaternionic (para)K\"ahler Heisenberg four-manifold is said to be stationary if $\varepsilon\Lambda=-6k^2$.
\end{mydef}
The name stationary comes from the fact that if we set $b=c$ and $h=0$ in the expression for the Ricci tensor given at Proposition \eqref{prop:riccitime}, then the Einstein condition reads:
\begin{eqnarray}
\label{eq:eins1}
\varepsilon \Lambda&=&\mu'+2 \lambda'-\mu^2-2\lambda^2\, , \\
\label{eq:eins2}
\varepsilon \Lambda &=&2 \frac{k^2 b^4}{a^2}- ( -\mu'+2\mu \lambda+\mu^2)  \, , \\
\label{eq:eins3}
\Lambda &=&-2 \varepsilon \frac{k^2 b^4}{a^2}-\varepsilon (-\lambda'+\lambda \mu+2\lambda^2)  \, , 
\end{eqnarray}
where $\mu=\log(a)'$ and $\lambda=\log(b)'$. If we set $\mu'=\lambda'=0$ we find that $\varepsilon \Lambda=-6k^2$, hence the name stationary. 
\begin{prop}
\label{prop:statsol}
All stationary (timelike) quaternionic (para)K\"ahler Heisenberg four-man-\\ifolds $(M,g)$ are given by:
\begin{equation}
a=e^{-2 k (t-t_0)}\, , \quad b=e^{- k (t-t_0)}\,.
\label{eq:statsol}
\end{equation}
They are  isometric to an open orbit of the solvable Iwasawa subgroup of $\mathrm{SU} (1,2)\cong \mathrm{SU}(2,1)$ on the symmetric space 
\begin{equation} \label{space:eq} \dfrac{\mathrm{SU}\left (\dfrac{3+\varepsilon}{2},\dfrac{3-\varepsilon}{2}\right )}{\mathrm{S}\left (\mathrm{U}(1) \times \mathrm{U}\left (\dfrac{3+\varepsilon}{2},\dfrac{1-\varepsilon}{2}\right )\right )},\end{equation} where $\mathrm{U}(p,q)$ denotes the (pseudo-)unitary group 
of the Hermitian sesquilinear form of index $q$. Moreover, when $(M,g)$ is Riemannian (resp. neutral signature) it is complete (resp. incomplete). 

\end{prop}
\begin{proof}
Since  $\mu'=\lambda'=0$ implies $\varepsilon\Lambda=-6k^2$, let us start by assuming $\mu'=\lambda'=0$. Consistency then requires:
\begin{equation}
\left ( \frac{b^4}{a^2} \right)'=0\,.
\end{equation}
This implies that $\mu =2 \lambda$. On substituting in \eqref{eq:eins1} we encounter:
\begin{equation}
-6 \lambda^2=\varepsilon \Lambda\,,
\end{equation}
and hence $\lambda=\pm   \sqrt{\frac{-\varepsilon \Lambda}{6}}=\pm k$ (remember that $k>0$). By the above, $\mu=\pm  2 k$. We observe in turn that the other equations \eqref{eq:eins2}, \eqref{eq:eins3} are satisfied and therefore the solution is:
\begin{equation}
a=e^{\pm 2 k (t-t_0)}\, , \quad b=e^{\pm k (t-t_0)}\,,
\end{equation}
where we have already imposed that $a(t_0)=b(t_0)=1$. Finally, we find that that Weyl tensor is self-dual only if we pick up the minus\footnote{Choosing the plus sign, the Weyl tensor is antiself-dual.} sign above,  so we check that these solutions are indeed stationary (timelike) quaternionic (para)K\"ahler four-manifolds.

\noindent
 In fact, solution \eqref{eq:statsol} exhausts the list of stationary (timelike) quaternionic (para)K\"ahler four-manifolds, and the argument to prove this statement goes as follows. First, we note that Proposition \ref{prop:solsbrt} guarantees that if $\varepsilon\Lambda>-81k^2$, then the cubic equation for $b'$ in \eqref{eq:edosqk} has a unique real solution given by $b'=\mathcal{B}_0$. Consequently, this implies that we obtain a system of ODEs of the form $a'=f_1(a,b,\varepsilon\Lambda)$ and $b'=f_2(a,b,\varepsilon\Lambda)$, where $f_1$ and $f_2$ are smooth functions (at least, for $\varepsilon\Lambda=-6k^2$) as long as $a, b \neq 0$. Since the solution \eqref{eq:statsol} verifies that $a, b \neq 0$ for all $t \in \mathbb{R}$, we observe that \eqref{eq:statsol} represents the unique solution to the ODEs $a'=f_1(a,b,\varepsilon\Lambda)$ and $b'=f_2(a,b,\varepsilon\Lambda)$, completing thus the classification of stationary (timelike) quaternionic (para)K\"ahler four-manifolds.

\noindent
On the other hand, after some computations we find that these configurations  satisfy that $\nabla \mathrm{R}^g=0$, where $\mathrm{R}^g$ is the Riemann curvature tensor of $g$. In other words, the resulting spaces are quaternionic K\"ahler in the Riemannian case and quaternionic paraK\"ahler in the  neutral-signature case. Comparing to the classification of pseudo-Riemannian symmetric spaces of quaternionic K\"ahler type, see \cite{W} and \cite{AC,Krahe}, we conclude (comparing curvature tensors) that the resulting spaces are locally isometric to the symmetric spaces (\ref{space:eq}). More precisely, the solutions are locally isometric to a left-invariant metric on the simply transitive solvable Iwasawa subgroup of $\mathrm{SU}(1,2)$ and $\mathrm{SU}(2,1)$ when $\varepsilon=1$ and $\varepsilon=-1$, respectively. To see this it suffices to observe that the Heisenberg group is included in a four-dimensional group of isometries, which is precisely the above-mentioned Iwasawa group. In fact, the one-parameter group
$t\mapsto t+t_0$, $x\mapsto e^{ t_0}x$, $y\mapsto e^{ t_0}y$, $z\mapsto e^{2 t_0}z$ acts by isometries, enlarging the Heisenberg group by a one-parametric group of automorphisms to the aforementioned solvable group. Finally, in the Riemannian case the metric is complete, since the interval $\mathcal{I}$ of definition of $g$ can be extended to the whole real line $\mathbb{R}$, while in the neutral-signature case the metric is (geodesically) incomplete, because  $\dfrac{\mathrm{SU}\left (1,2\right )}{\mathrm{S}\left (\mathrm{U}(1) \times \mathrm{U}\left (1,1\right )\right )}$ is homotopically equivalent to $\mathrm{S}^2$ and hence it cannot be diffeomorphic to $\mathbb{R}^4$.
\end{proof}
\begin{rem}
We have not used  the second equation in \eqref{eq:edosqk} together with equation \eqref{eq:desbriem1} because it was easier to directly obtain the stationary solutions from the Ricci tensor given at Proposition \eqref{prop:riccitime}. However, it is worth noting that the solution $a=e^{-2k(t-t_0)}$, $b=e^{-k(t-t_0)}$ does indeed solve the second equation in \eqref{eq:edosqk} and \eqref{eq:desbriem1}.
\end{rem}

\begin{rem}
In the Riemannian case, it is possible to see that the solution \eqref{eq:statsol} is completely equivalent to that obtained in \cite[Proposition 4.1]{CS}. Using the subindex CS to make reference to quantities of that article, by performing the identifications $e_{\mathrm{CS}}^i=k \mathfrak{e}_{t_0}^i$, $a_{\mathrm{CS}}(t)=k^{-2} a^{-2}(t)$, $b_{\mathrm{CS}}(t)=k^{-2} b^{-2}(t) $, $\mu_{\mathrm{CS}}=2 k$ and  $C_{\mathrm{CS}}=k^{-2} e^{-2kt_0}$, we conclude that our stationary quaternionic K\"ahler Heisenberg four-manifolds are equivalent to the stationary solutions of \cite{CS}.
\end{rem}

\subsubsection{Negative (timelike) quaternionic (para)K\"ahler Heisenberg four-manifolds}

\begin{mydef}
A (timelike) quaternionic (para)K\"ahler Heisenberg four-manifold is said to be negative if $\varepsilon\Lambda<0$ with  $\varepsilon\Lambda\neq -6 k^2$.
\end{mydef}
\begin{rem}
Observe that negative quaternionic K\"ahler Heisenberg four-manifolds have $\Lambda <0$ while the negative timelike quaternionic paraK\"ahler ones have $\Lambda >0$.  
\end{rem}
Let $\gamma \in \mathbb{R}$ and let $I$ be a connected component of the set:
\begin{equation}
\{ \rho \in \mathbb{R}\, \vert\, \rho \neq 0 , \rho+\gamma>0 \, \, \text{and}\, \, \rho+2\gamma>0\}\,.
\label{eq:intdefrho}
\end{equation}
Let $\rho : J \overset{\sim}{\rightarrow} I$, $t \mapsto \rho(t)$ be a maximal solution of the ordinary differential equation:
\begin{equation}
\rho'(t)=\sqrt{-\frac{2 \varepsilon\Lambda }{3}} \rho(t) \sqrt{\frac{\rho(t)+\gamma}{\rho(t)+2\gamma}}\,,
\label{eq:defrho}
\end{equation}
with the initial condition $\rho(0)=\rho_0$. Define:
\begin{equation}
\label{eq:absols}
A_{s}(\rho_0,\gamma)=s k \sqrt{-\frac{2 \varepsilon\Lambda }{3}} \rho(t) \sqrt{\frac{\rho(t)+2\gamma}{\rho(t)+\gamma}}\,, \quad  B_{s}(\rho_0,\gamma)=s \sqrt{-\frac{\varepsilon\Lambda }{3}} \frac{\rho(t)}{\sqrt{\rho(t)+2\gamma}}\,,
\end{equation}
with $s \in \mathbb{Z}_2$ a sign.
\begin{prop}
\label{prop:inivalues}
Let $(A_{s}(\rho_0,\gamma),B_{s}(\rho_0,\gamma))$ as in \eqref{eq:absols}. On the one hand, if $s=1$ and $\varepsilon\Lambda <0$, then there exists a unique pair $(\rho_1,\gamma_1)$ such that:
\begin{equation}
\label{eq:eqsci1}
A_{s}(\rho_1,\gamma_1)=B_{s}(\rho_1,\gamma_1)=1\,.
\end{equation}
On the other hand, if $s=-1$ and $\varepsilon\Lambda \leq -81 k^2$ there exist two pairs of solutions $(\rho_2,\gamma_2)$ $(\rho_3,\gamma_3)$ such that:
\begin{equation}
\label{eq:eqsci2}
A_{s}(\rho_2,\gamma_2)=B_{s}(\rho_2,\gamma_2)=1\,, \quad A_{s}(\rho_3,\gamma_3)=B_{s}(\rho_3,\gamma_3)=1\,.
\end{equation}
Such initial conditions $(\rho_l,\gamma_l)$ with $l=1,2,3$ are given by:
\begin{align}
\rho_l&=-\frac{e^{-(4-2l)i \pi/3}}{2k (\varepsilon \Lambda)^{1/3}(9k+\sqrt{81k^2+ \varepsilon\Lambda})^{1/3}}+e^{(4-2l)i \pi/3} \frac{(9k+\sqrt{81k^2+ \varepsilon\Lambda})^{1/3}}{2k (\varepsilon\Lambda)^{2/3}}\, ,\\
2\gamma_l&=-\rho_l\left ( 1+\frac{\varepsilon \Lambda}{3} \rho_0\right)\,.
\end{align}
\end{prop}
\begin{proof}
Demanding $A_{s}(\rho_0,\gamma)=B_{s}(\rho_0,\gamma)=1$, we can solve for $\gamma$ in the last equation obtaining:
\begin{equation}
2\gamma=-\rho_0\left (1+\frac{\varepsilon\Lambda}{3}\rho_0 \right )\,.
\label{eq:gamma}
\end{equation}
Squaring the equation $A_{s}(\rho_0,\gamma)=1$ and substituting this result for $\gamma$, we arrive at the following cubic polynomial for $\rho_0$:
\begin{equation}
\rho_0^3+\frac{3}{4 \varepsilon \Lambda k^2}\rho_0-\frac{9}{4 k^2 \Lambda^2}=0\,.
\label{eq:cubicpol}
\end{equation} 
The (maybe complex) solutions to this cubic equations are:
\begin{equation}
\rho_l=-\frac{e^{-(4-2l)i \pi/3}}{2k (\varepsilon \Lambda)^{1/3}(9k+\sqrt{81k^2+ \varepsilon\Lambda})^{1/3}}+e^{(4-2l)i \pi/3} \frac{(9k+\sqrt{81k^2+ \varepsilon\Lambda})^{1/3}}{2k (\varepsilon\Lambda)^{2/3}}\, , \quad l=1,2,3\,.
\label{eq:solsprueba}
\end{equation}
At least one of the previous solution is real. However, not all real solutions of these equations need to satisfy a posteriori $A_{s}(\rho_0,\gamma)=B_{s}(\rho_0,\gamma)=1$, since in the process of arriving \eqref{eq:cubicpol} one has squared some expressions. Let us split this analysis between the cases $s=\pm 1$:
\begin{itemize}
\item If $s=1$  we find that there is a unique real solution of \eqref{eq:solsprueba} which in turn satisfies \eqref{eq:eqsci1} for all values of $\varepsilon\Lambda <0$. This solution is the one in \eqref{eq:solsprueba} for $l=1$, that we denote as $\rho_1$. Substituting this expression of $\rho_1$ in \eqref{eq:gamma} we obtain the unique solution $(\rho_1,\gamma_1)$. Also, \emph{a posteriori} we check that $\rho_1+2\gamma_1>0$, $\rho_1+\gamma_1> 0$ and $\rho_1\neq 0$ for all $\varepsilon \Lambda <0$. On varying the value of $\varepsilon\Lambda$, we observe by direct inspection that $\rho_0 \in (0,+\infty)$, while $\gamma \in (-\infty, (8k^2)^{-1})$. Interestingly, $\gamma$ is negative when $0>\varepsilon\Lambda >-6k^2$ and positive if  $\varepsilon\Lambda <-6k^2$.


\item If $s=-1$, we find interestingly enough that no real solutions exist (after checking if they satisfy \eqref{eq:eqsci2}) for $\varepsilon \Lambda >-81 k^2$, while for $\varepsilon\Lambda \leq -81 k^2$ we have two possible solutions\footnote{They actually coincide for $\varepsilon\Lambda =-81k ^2$.}. These solutions are the ones in \eqref{eq:solsprueba} with $l=2$ and $l=3$, that we denote respectively as $\rho_2$ and $\rho_3$. Substituting them in \eqref{eq:gamma} we obtain two solutions $(\rho_2,\gamma_2)$ and $(\rho_3,\gamma_3)$. We check \emph{a posteriori} that both satisfy $\rho_l+2\gamma_l>0$, $\rho_l+\gamma_l> 0$ and $\rho_l\neq 0$ with $l=2,3$ for all $\varepsilon \Lambda \leq 81 k^2$. On the other hand, for the different values of $\varepsilon\Lambda \leq -81 k^2$, by direct inspection we see that $\rho_2 \in  ( -(18 k^2)^{-1}, 0)$ and $\rho_3 \in \left (- \frac{5}{80 k^2},0 \right )$, while $\gamma_2 \in \left ( 0, \frac{5}{72 k^2} \right) $ and $\gamma_3 \in \left( \frac{5}{72 k^2}, (8k^2)^{-1} \right)$.
\end{itemize}
\end{proof}

\noindent 
According to Proposition \ref{prop:solsbrt}, there exists a unique solution for the pair $(a(t),b(t))$ if $\varepsilon\Lambda>-81 k^2$, while there are three (some of them identical for $\varepsilon\Lambda=-81 k^2$) if $\varepsilon\Lambda\leq -81 k^2$. By use of Proposition \ref{prop:inivalues}, it is possible to find such solutions, which we write next.
\begin{prop}\label{prop:neg(time)q(para)}
Let $(\rho_l,\gamma_l)$ for $l=1,2,3$ denote the pairs of Proposition \ref{prop:inivalues} and let $(A_s (\rho_l,\gamma_l),B_s(\rho_l,\gamma_l))$ be as in \eqref{eq:absols}. Set:
\begin{eqnarray}
\label{eq:uhm}
(a_1(t),b_1(t))&=&(A_{1}(\rho_1,\gamma_1),B_{1}(\rho_1,\gamma_1))\,,\\
\label{eq:uhm2}
(a_2(t),b_2(t))&=&(A_{-1}(\rho_2,\gamma_2),B_{-1}(\rho_2,\gamma_2))\,,\\
\label{eq:uhm3}
(a_3(t),b_3(t))&=&(A_{-1}(\rho_3,\gamma_3),B_{-1}(\rho_3,\gamma_3))\,.
\end{eqnarray}
These are all the solutions to \eqref{eq:edosqk} (or \eqref{eq:desbriem1}) and, consequently, all negative (timelike) quaternionic (para)K\"ahler Heisenberg four-manifolds. In particular, $(a_1(t),b_1(t))$ is defined for all $\varepsilon \Lambda <0$ while $(a_2(t),b_2(t))$ and $(a_3(t),b_3(t))$ are defined for all $\varepsilon \Lambda \leq -81 k^2$. The corresponding pseudo-Riemannian manifolds arising from \eqref{eq:uhm2} and \eqref{eq:uhm3}, together with those stemming from \eqref{eq:uhm} for $\varepsilon\Lambda >-6k^2$, are incomplete, while in the case \eqref{eq:uhm} for $\varepsilon=1$ and $\Lambda<6k^2$ the solution is complete.  
\end{prop}
\begin{proof}
The fact that \eqref{eq:uhm}, \eqref{eq:uhm2} and \eqref{eq:uhm3} solve equations \eqref{eq:edosqk} (or \eqref{eq:desbriem1}) with the initial condition $a(t_0)=b(t_0)=1$ follows from direct computation and by Proposition \ref{prop:inivalues}. Regarding completeness, we go on a case by case fashion:
\begin{itemize}
\item Let us begin by analyzing solutions \eqref{eq:uhm2} and \eqref{eq:uhm3}. In these cases, we have $\rho_0<0$, $\gamma>0$ but $\rho_0+\gamma>0$. We consider the canonical geodesic defined by the coordinate $\rho$ (related to the time coordinate as in \eqref{eq:defrho}), with $\rho$ defined between $(-\gamma,\rho_0)$. We compute its length:
\begin{equation*}
\sqrt{\left \vert \frac{3}{2 \varepsilon\Lambda} \right \vert} \int_{-\gamma}^{\rho_0} \frac{1}{\vert \rho \vert} \sqrt{\left \vert \frac{\rho+2\gamma}{\rho+\gamma} \right \vert}\d \rho\leq C \int_0^{\rho_0+\gamma} \frac{1}{\sqrt{\tau}}\ d \tau< \infty\,,
\end{equation*}
where $C>0$ is given by $C=\sqrt{\left \vert \frac{3}{2 \varepsilon\Lambda} \right \vert} (\rho_0+\gamma) \zeta$, being $\zeta$ the maximum of the function $\frac{\sqrt{\rho+2\gamma}}{\vert \rho \vert}$ on the compact interval $[-\gamma,\rho_0]$. Since the length of this curve, which arrives to the boundary of the domain definition of the parameter $\rho$, is finite, we conclude that  solutions \eqref{eq:uhm2} and \eqref{eq:uhm3} are incomplete.     
\item For solutions \eqref{eq:uhm} with $-6k^2<\varepsilon\Lambda<0$, by virtue of Proposition \ref{prop:inivalues} and its proof we realize that $\gamma<0$  (but again, $\rho_0>-\gamma$). Therefore, through a completely equivalent proof to that provided in the previous bullet-point, we observe that these solutions are incomplete too.   
\item As explained in Remark \ref{rem:reluhmvc} below, solutions \eqref{eq:uhm}, \eqref{eq:uhm2} and \eqref{eq:uhm3} with $\varepsilon=1$  are identified with the one-loop deformed universal hypermultiplet metrics (see Remark \ref{rem:reluhmvc}) described in \cite{AMTV,RSV}. They are known to be complete if and only if $\gamma$ and the initial condition $\rho_0$ are positive \cite{ACDM}. For $\varepsilon=1$ and $\Lambda<-6 k^2$, we observe that Proposition \ref{prop:inivalues} and its proof ensure that $\gamma$ and $\rho_0$ are positive for \eqref{eq:uhm}, and consequently we infer that they are complete.  
\end{itemize}
\end{proof}
\begin{rem}
\label{rem:incomplequizas}
We strongly believe the case \eqref{eq:uhm} with $\varepsilon=-1$ and $\Lambda >6k^2$ to be incomplete as well. Indeed, let us use the coordinates \eqref{eq:coordh} to describe the Heisenberg group $\mathrm{H}$. Let us consider a geodesic $\Gamma: J \rightarrow (\mathcal{I}\times \mathrm{H})$ for $J \subset \mathbb{R}$ whose coordinates are given by $(\rho(\tau), x(\tau),y(\tau),z(\tau))$ with $\tau \in J$ an affine parameter and with the initial conditions $\rho(0)=\rho_0>0$, $x(0)=y(0)=z(0)=0$, $y'(0)=z'(0)=0$ and $x'(0)=v_0$. On the one hand, we find that $y(\tau)=z(\tau)=0$. On the other hand, the solution for $x(\tau)$ can be seen to be:
\begin{equation*}
x(\tau)=\kappa \int_{t_0}^\tau  \frac{\rho^2(\sigma)}{\rho(\sigma)+2\gamma} \d \sigma\, ,
\end{equation*}
where $\kappa \in \mathbb{R}$ is some constant ensuring that $x'(0)=v_0$. Using this result, the equation for $\rho=\rho(\tau)$ turns out to be:
\begin{equation*}
\rho''=\frac{v_0^2 \rho^3 (\rho+\gamma)(\rho+4 \gamma) }{(\rho+2\gamma)^3}+\frac{(\rho')^2(4 \gamma^2 +7 \gamma \rho+2 \rho^2)}{2 \rho(\rho+\gamma)(\rho+2\gamma)}\,.
\end{equation*}
By numerical analysis, it can be seen that the solutions for the previous second-order differential equation equation are, typically, only defined in a finite interval $J$. Hence this provides a robust argument in favour of the incompleteness of the solutions.  
\end{rem}
\begin{rem}
\label{rem:reluhmvc}
Let us show that the solution \eqref{eq:uhm}  for $\varepsilon=1$ is equivalent up to homothety to the one-loop deformed universal hypermultiplet described in \cite{CS}. For that, let us denote the quantities of that work by a superscript or a subindex $\CS$. Following their notation, we define $\d t_{\CS}=\sqrt{-\frac{2\Lambda}{3}} \d t$, rescale their metric with the factor $-\frac{2\Lambda}{3}$ (remember $\Lambda <0$) and define $e^3_{\CS}=\frac{1}{k w} \fre_{t_0}^1$, $e^1_{\CS}=\frac{1}{k \sqrt{w}} \fre_{t_0}^2$ and $e^2_{\CS}=\frac{1}{k \sqrt{w}} \fre_{t_0}^3$ with $w > 0$ such that $c_{\CS}=w \gamma$. Note that the corresponding dual vectors satisfy\footnote{Therefore we may use the coordinates \eqref{eq:coordh} and their dual basis of one-forms \eqref{eq:dualbasisheis} for $e^i_{\CS}$ by setting $k=1$.}$[e_2^{\CS},e_3^{\CS}]=-2 e_1^{\CS}$. We also set  $a_{\CS}=-2 \Lambda k^2 w^2 a^{-2} /3$ and $b_{\CS}=-2 \Lambda k^2 w b^{-2} /3$ and rescale $\rho=w \rho_{CS}$. After all these identifications, we still have the freedom to perform time shifts on $t_{\CS}$. In order to match  $c_{\CS}>0$ and the initial condition for $\rho_{CS}$, $\rho_{CS}^0$:
\begin{itemize}
\item If $c_{\CS}>0$ and $\rho_{CS}^0>0$, then we recover our solution \eqref{eq:uhm} with $\Lambda <-6 k^2$. To see this, we note that in this case $0<\gamma<(8k^2)^{-1}$ (as  Proposition \ref{prop:inivalues} and its proof reveal). Consequently after an appropriate time shift and choosing $w=\tan\left (4 \pi k^2 \gamma \right)$, we observe that we may obtain all possible $c_{\CS}>0$ and $\rho_{CS}^0>0$.
\item If $c_{\CS}<0$ and $\rho_{CS}^0>0$, then they correspond to the solution \eqref{eq:uhm} with $0>\Lambda >-6 k^2$. The identification with $(c_{\CS},\rho_{CS}^0)$ is obtained by setting $w=1$ and performing a suitable time shift on $t_{\CS}$.
\item If $c_{\CS}>0$ and $\rho_{CS}^0<0$, we distinguish two different cases. Set $w=\tan\left (4\pi k^2 \gamma \right)$ as before. Then we identify (Riemannian) solutions \eqref{eq:uhm2}  with $c_{\CS} \in \left (0, \frac{5}{72 k^2} \tan\left( \frac{5 \pi}{18} \right) \right)$  and \eqref{eq:uhm3} with $c_{\CS} \in \left (\frac{5}{72 k^2} \tan\left( \frac{5 \pi}{18} \right), +\infty \right)$ after performing time translations, if necessary. If  $c_{\CS}=\frac{5}{72 k^2} \tan\left( \frac{5 \pi}{18} \right)$, then the corresponding $\gamma$ is the one for which \eqref{eq:uhm2} and \eqref{eq:uhm3} coincide and we can take any of them.
\end{itemize}
\end{rem}
\begin{rem}
Regarding the $\varepsilon=-1$ case, it can be seen that they correspond to a neutral-signature version of the one-loop deformed universal hypermultiplet. More concretely, these negative timelike quaternionic paraK\"ahler Heisenberg four-manifolds can be obtained through the local temporal supergravity $c$-map \cite{CMMS2}. In order to see this, we just have to start from the trivial zero-dimensional manifold $\bar{M}$ given by a point and set in Equation (59) of Reference \cite{DV} (following their notation) $\epsilon_1=-\epsilon_2=-1$, $I=0$, $z^0=1$ and $F_0=\frac{i_{\epsilon_1}}{2} (X^0)^2$, what implies in turn that $e^{\mathcal{K}}=1/2$ and $(\hat{H}^{ab})=\mathrm{diag}(1,1)$. Finally, by a completely analogous procedure to that described in Remark \ref{rem:reluhmvc}, we observe that we get, up to a global sign, our negative timelike quaternionic paraK\"ahler Heisenberg four-manifolds. 

\end{rem}

\subsubsection{Positive (timelike) quaternionic (para)K\"ahler Heisenberg four-manifolds}

\begin{mydef}
A (timelike) quaternionic (para)K\"ahler Heisenberg four-manifold is said to be positive if $\varepsilon\Lambda>0$.
\end{mydef}
\begin{rem}
Note that positive quaternionic K\"ahler Heisenberg four-manifolds have $\Lambda >0$ while the positive timelike quaternionic paraK\"ahler ones have $\Lambda <0$. 
\end{rem}

Let $\gamma \in \mathbb{R}$ and let $I$ be a connected component of the set:
\begin{equation}
\{ \rho \in \mathbb{R}\, \vert\, \rho \neq 0 , \rho+\gamma>0 \, \,  \text{and}\,  \, \rho+2\gamma<0 \,  \}\,.
\label{eq:intdefrhopos}
\end{equation}
Clearly, $\rho>0$ and $\gamma <0$. Let $\rho : J \overset{\sim}{\rightarrow} I$, $t \mapsto \rho(t)$ be a maximal solution of the ordinary differential equation:
\begin{equation}
\rho'(t)=\sqrt{\frac{2 \varepsilon\Lambda}{3}} \rho(t) \sqrt{-\frac{\rho(t)+\gamma}{\rho(t)+2\gamma}}\,,
\label{eq:defrhopos}
\end{equation}
with initial condition $\rho(0)=\rho_0$. Define:
\begin{equation}
\label{eq:abpos}
A(\rho_0,\gamma)=k \sqrt{\frac{2 \varepsilon \Lambda}{3}} \rho(t) \sqrt{-\frac{\rho(t)+2 \gamma}{\rho(t)+\gamma}}\, , \quad B(\rho_0,\gamma)=\sqrt{\frac{\varepsilon \Lambda}{3}} \frac{\rho(t)}{\sqrt{-\rho(t)-2\gamma}}\,.
\end{equation}
\begin{prop}
\label{prop:inivaluepos}
Let $(A(\rho_0, \gamma),B(\rho_0,\gamma))$ as in \eqref{eq:abpos}. If $\varepsilon\Lambda >0$, there exists a unique pair $(\rho_0,\gamma)$ such that:
\begin{equation}
A(\rho_0, \gamma)=B(\rho_0,\gamma)=1\,.
\end{equation}
\end{prop}
\begin{proof}
If we impose $A(\rho_0, \gamma)=B(\rho_0,\gamma)=1$, we can solve for $\gamma$ in the equation $B(\rho_0,\gamma)=1$ and get:
\begin{equation}
\label{eq:gammapos}
2\gamma=-\rho_0\left (1+\frac{\varepsilon\Lambda}{3}\rho_0 \right)\,.
\end{equation}
Now substituting this result in $A(\rho_0, \gamma)=1$ we find the following cubic equation:
\begin{equation}
\rho_0^3+\frac{3}{4 \varepsilon \Lambda k^2} \rho_0-\frac{9}{4k^2 \Lambda^2}=0\,.
\end{equation}
This is formally equivalent to that found for negative (timelike) quaternionic (para)K\"ahler Heisenberg four-manifolds, although now $\varepsilon \Lambda >0$. The previous cubic equation has a unique real root when $\varepsilon \Lambda >0$ and it turns out to be:
\begin{equation}
\rho_0=\frac{-1+ (\varepsilon\Lambda)^{-1/3} (9k+\sqrt{81k^2 +\varepsilon\Lambda})^{2/3}}{ 2k (\varepsilon\Lambda)^{1/3} (9k+\sqrt{81k^2 +\varepsilon\Lambda})^{1/3}}>0\,.
\end{equation}
Taking this expression into that of $\gamma$ given at \eqref{eq:gammapos}, we find the unique solution $(\rho_0,\gamma)$. Finally, after a careful study, in addition to $\rho_0 \neq 0$ we may learn that $\rho_0+\gamma >0$ and $\rho_0+2\gamma<0$ for all $\varepsilon\Lambda >0$ and we conclude.
\end{proof}

\noindent
According to Proposition \ref{prop:solsbrt}, there exists a unique solution $(a(t),b(t))$ for $\varepsilon\Lambda>0$, which we proceed to present now.

\begin{prop}\label{prop:pos(time)q(para)}
Let $(\rho_0,\gamma)$  denote the pair of Proposition \ref{prop:inivaluepos}
and $(A(\rho_0,\gamma),B(\rho_0,\gamma))$ as in \eqref{eq:abpos}. Set:
\begin{eqnarray}
\label{eq:uhmpos}
(a(t),b(t))&=&(A(\rho_0,\gamma),B(\rho_0,\gamma))\,,
\end{eqnarray}
These are all the solutions to \eqref{eq:edosqk} with $\varepsilon\Lambda >0$ and, consequently, all positive (timelike) quaternionic (para)K\"ahler Heisenberg four-manifolds. They all are incomplete. 
\end{prop}
\begin{proof}
The fact that they all are solutions to \eqref{eq:edosqk} with $\varepsilon\Lambda >0$ follows by direct computations and by use of Proposition \ref{prop:inivaluepos}. Regarding the incompleteness, since $\rho_0>0$ and $\gamma<0$,  let us consider the  geodesic defined by the coordinate $\rho$ (related to the temporal coordinate as in \eqref{eq:defrho}), with $\rho$ defined between $(-\gamma,-2 \gamma)$. We calculate its length:
\begin{equation*}
 \sqrt{\frac{3}{2 \varepsilon\Lambda}} \int_{-\gamma}^{-2\gamma} \frac{1}{\vert \rho \vert} \sqrt{\left \vert \frac{\rho+2\gamma}{\rho+\gamma} \right \vert}\d \rho\leq C \int_0^{-\gamma} \frac{1}{\sqrt{\tau}}\ d \tau< \infty\,,
\end{equation*}
where $C>0$ is given by $C=- \sqrt{\frac{3 }{2 \varepsilon\Lambda}} \gamma \zeta$, being $\zeta$ the maximum of the function $\frac{\sqrt{\vert \rho+2\gamma\vert }}{\rho}$ on the compact interval $[-\gamma,-2\gamma]$. Since the length of this curve, which reaches the boundary of the domain of definition of the parameter $\rho$, is finite, we conclude that the solution \eqref{eq:uhmpos} is incomplete.

\end{proof}
\begin{rem}
The non-stationary solutions with $\varepsilon\Lambda >0$ are obtained from those with $\varepsilon\Lambda <0$ basically by replacing $\varepsilon\Lambda \rightarrow -\varepsilon\Lambda$ and changing suitably the domain of $\rho$ to ensure the reality of the coordinate.
\end{rem}
\begin{rem}
We may interpret positive (timelike) quaternionic (para)K\"ahler Heisenberg four-manifolds as positively (negatively) curved versions of the one-loop deformed universal hypermultiplet solution.
\end{rem}

\noindent
\textbf{Proof of Theorem \ref{thm:alltimeqknonstat}}.  Collecting the results given in Propositions \ref{prop:inivalues}, \ref{prop:neg(time)q(para)}, \ref{prop:inivaluepos} and \ref{prop:pos(time)q(para)} and expressing the metric in terms of the coordinate $\rho$, as defined by \eqref{eq:intdefrho} and \eqref{eq:defrho} in the negative case and by \eqref{eq:intdefrhopos} and \eqref{eq:defrhopos} in the positive case, we conclude.

\subsection{Spacelike quaternionic paraK\"ahler Heisenberg four-manifolds}

Now we continue with the classification of all spacelike quaternionic paraK\"ahler Heisenberg four-manifolds. For that, we shall use the ansatz described in \eqref{eq:ansspacesim},  valid in an open subinterval $\mathcal{I}'_s \subset \mathcal{I}$, which we rewrite here for the benefit of the reader:
\begin{equation}
\label{eq:anssiem2}
U^t=\begin{pmatrix}
c(t) &h(t) & 0\\
-h(t) & b(t) & 0 \\
0 & 0 & a(t) \\
\end{pmatrix}\,, \quad [\fre_1^{t_0},\fre_2^{t_0}]=-2k \fre_3^{t_0}\,, \quad k>0\,.
\end{equation}
\begin{prop}
\label{prop:riccispace}
The non-zero components of the Ricci curvature tensor $\ric^g$ of the metric obtained from \eqref{eq:anssiem2} are:
\begin{eqnarray*}
\mathrm{Ric}^g(\partial_t ,\partial_t)&=& \frac{(-2a^2ch (7 b'+c') h'+a^2 c^2(-4(b')^2+(h')^2)+h^4(-4(a')^2+2 aa''))}{2 a^2(bc+h^2)^2}\\
& &+\frac{a^2 h^2((b')^2+6 b' c'+(c')^2-8 (h')^2+2c b'')+b^2 a^2(-4(c')^2+(h')^2)}{2 a^2(bc+h^2)^2}\\ & & \frac{+b^2 (c^2(-4 (a')^2+2a a'')+2a^2 cc''))+4a^2 h^3 h''+2b a^2 c^2 b''-2 ba^2 h b' h'}{2 a^2(bc+h^2)^2}\\ & & +\frac{2b (a^2 h(-7 c'h'+h c'')+c(3a^2(h')^2+h^2(-4(a')^2+2 a a'')+2 a^2 h h'')  )}{2 a^2(bc+h^2)^2}\, , \\
  \mathrm{Ric}^g(\mathfrak{e}_1^t,\mathfrak{e}_1^t)&=&-2 k^2 \frac{(bc+h^2)^2}{a^2}-\frac{2bc b'c'+4 b^2 (c')^2+b^2 (h')^2-2bc (h')^2-c^2(h')^2}{2(bc+h^2)^2}\\& & -\frac{a'(bc'+hh')}{a(bc+h^2)}+\frac{2b^2 c c''+h^2(-(b')^2+2 b'c'+(c')^2-6(h')^2+2b c''))}{2(bc+h^2)^2}\\ & & +\frac{2 h^3 h''+2h((b'(b-2c)-c'(6b+c))h'+bc h'')}{2(bc+h^2)^2}\, , \\
   \mathrm{Ric}^g(\mathfrak{e}_1^t,\mathfrak{e}_2^t) &=& \frac{3 b^2 c' h'}{2(bc+h^2)^2}+\frac{a'(-h(b'+c')+(b+c)h')}{2a(bc+h^2)}+\frac{h^2(h(b''+c'')-4(b'+c')h')}{2(bc+h^2)^2}\\& &+\frac{c^2(3b'h'-bh'')-bh(3b'c'+(c')^2-4(h')^2+h h'')-chb' (b'+3c')}{2(bc+h^2)^2} \\ & & +\frac{ch(4(h')^2+b(b''+c''))-c(h^2 h''+b((b'+c')h'-bh''))}{2(bc+h^2)^2}
   \, , \\
\mathrm{Ric}^g(\mathfrak{e}_2^t,\mathfrak{e}_2^t)&=&2 k^2 \frac{(bc+h^2)^2}{a^2}+\frac{h'((c^2-b^2+6 h^2-2bc)h'+12chb'-2c c' h+2 bh(b'+2c')))}{2(bc+h^2)^2}\\
& & +\frac{a'(cb'+hh')}{a(bc+h^2)}+\frac{b'(4c^2 b'-h^2 b'-2h^2c'+2 b c c')+h^2 (c')^2}{2(bc+h^2)^2} \\ & & -\frac{b''(ch^2+bc^2)+h''h (h^2+bc)}{(bc+h^2)^2}\, ,\\
\mathrm{Ric}^g(\mathfrak{e}_3^t,\mathfrak{e}_3^t)&=&-2k^2\frac{(bc+h^2)^2}{a^2}+\frac{2 (a')^2}{a^2}+ \frac{a'(cb'+bc'+2h h')}{a(bc+h^2)}-\frac{a''}{a} \, .\\
\end{eqnarray*}
\end{prop}
\begin{proof}
By direct computation. 
\end{proof}
\begin{prop}
\label{prop:eqsspace}
Let $(M,g)$ be a spacelike quaternionic paraK\"ahler Heisenberg four-manifold. Then:
\begin{equation}
\label{eq:edoskspace}
a'(t)=-\frac{k^2 b^6+a^2(-\Lambda b^2+(b')^2)}{2ab b'}\, , \quad \Lambda=\frac{3(kb^3+a b')^3}{a^2 b^2(3k b^3+a b')}\, , \quad c=b\, , \quad h=0\,. \\
\end{equation}
Furthermore, $\mathcal{I}'_s=\mathcal{I}$.
\end{prop}
\begin{proof}
We remind that a spacelike quaternionic para K\"ahler Heisenberg four-manifold is Einstein and its Weyl tensor is self-dual. On the one hand, by setting that $\ric^g(\fre_1^t,\fre_2^t)=0$ one infers:
\begin{eqnarray}
\nonumber 
h''&=&\frac{3 a c^2 b' h'+b^2(ca'+3a c')h'+h^2(ca'-4a(b'+c')) h'-ach((b')^2+3b'c'-4(h')^2)}{a(b+c)(bc+h^2)}\\ \nonumber & &+\frac{h^3(a(b''+c'')-a'(b'+c'))+b(c^2a'h'-h a c'(3b'+c')+h^2a' h'+4 a h (h')^2)}{a(b+c)(bc+h^2)}\\ & &+\frac{bc(a (b'+c')h'+h(a(b''+c'')-a'(b'+c'))) }{a(b+c)(bc+h^2)} \,.
\end{eqnarray}
Computing the Weyl self-duality tensor $\mathcal{W}^g$ for the ansatz \eqref{eq:anssiem2} and substituting the previous result, we have that:
\begin{equation}
\mathcal{W}(\fre_1^t,\partial_t ,\fre_3^t,\fre_1^t)=\frac{(kbc+kh^2-a')(-h(b'+c')+(b+c)h')}{2a (bc+h^2) }=0\,.
\end{equation}
If $a'=k (bc+h^2)$, this would in turn imply that $\Lambda=0$, but we are imposing the condition that $\Lambda \neq 0$ to study proper quaternionic paraK\"ahler four-manifolds, so we conclude that necessarily $-h(b'+c')+(b+c)h'=0$. Since $h(t_0)=0$ and $b(t_0)+c(t_0)=2$, there exists an open subinterval of $\mathcal{I}'_s$ in which $b(t)+c(t) \neq 0$ and $h(t)=0$. In fact, $b(t)+c(t) \neq 0$ for all $t \in \mathcal{I}'_s$, because otherwise  there would be $t_1 \in \mathcal{I}'_s$ satisfying $b(t_1)=-c(t_1)$ and $h(t_1)=0$. This would imply in turn that $\det U^{t_1}=-a(t_1) b(t_1)^2$. If $a(t_1)<0$, we see that there was a time $t_2 \in \mathcal{I}'_s$ in which $a(t_2)=\det U^{t_2}=0$ (remember that $a(t_0)=1)$ and if $a(t_1)>0$, then $\det U^{t_1}<0$ and there was a time $t_3$ in which $\det U^{t_3}=0$ (note that $\det U^{t_0}=1$). Since $U^t$ is non-degenerate by hypothesis for all $t \in \mathcal{I}'_s$, we conclude that $h(t)=0$ and $b(t)+c(t) \neq 0$ in the entire $\mathcal{I}'_s$.

Now we can solve for $a'', b'', c''$ and $a'$ from the Einstein condition $\ric^g=\Lambda g$ and we get:
\begin{eqnarray}
a'(t)&=&-\frac{k^2 b^3 c^3+a^2(-\Lambda b c+b' c')}{a(cb'+b c')}\, , \\
\nonumber
a''(t)&=&\frac{-k^2 a^2 b^3 c^3 \left(2 b c b' c'+3 c^2 (b')^2+b^2 \left(3 (c')^2+4 \Lambda  c^2\right)\right)}{a^3 b c \left(c b'+b c'\right)^2}\\ \label{eq:secondaspace} &+&\frac{a^4 \left(-b^2 b' c' \left((c')^2+4 \Lambda  c^2\right)-c^2 (b')^3 c'+2 \Lambda ^2 b^3 c^3\right)+2 k^4 b^7 c^7}{a^3 b c \left(c b'+b c'\right)^2}\,,\\
\label{eq:secondbspace}
b''(t) &=& \frac{a^2 b^2 b' (c')^2+2 a^2 b c (b')^2 c'+2 a^2 c^2 (b')^3-\Lambda  a^2 b^3 c c'+k^2 b^4 c^4 b'+2 k^2 b^5 c^3 c'}{a^2 b c \left(c b'+b c'\right)}\,,\\
\label{eq:secondcspace}
c''(t) &=& \frac{a^2 \left(b c b' \left(2 (c')^2-\Lambda  c^2\right)+c^2 (b')^2 c'+2 b^2 (c')^3\right)+k^2 b^3 c^4 \left(2 c b'+b c'\right)}{a^2 b c \left(c b'+b c'\right)}\,.
\end{eqnarray}
Taking these results into the remaining components of the Weyl self-duality tensor, we find in particular that:
\begin{eqnarray}
\nonumber
\mathcal{W}^g(\fre_2^{t},\fre_3^{t},\fre_2^{t},\fre_3^{t})&=&-k^3\frac{b^4 c^4}{a^3 (cb'+bc')}- k^2 \frac{b^2 c^2(2cb'+b c')}{a^2(cb'+bc')}-k c \frac{2 bb'c'+c (b')^2-\Lambda b^2 c}{a (cb'+bc')}\\ & &-\frac{c'(\Lambda b^2+3 (b')^2)-2\Lambda b c b'}{3b (cb'+bc')}\,.
\end{eqnarray}
Equating the last expression to zero, one solves for $\Lambda$ and finds:
\begin{equation*}
\Lambda=-\frac{3(kb^2 c+a b')^2(kbc^2+a c')}{a^2 b(-3k b^2 c^2-2a c b'+a b c')}\,.
\end{equation*}
Plugging this last result into the other components of $\mathcal{W}^g$, we encounter in particular:
\begin{equation}
\mathcal{W}^g(\fre_1^{t},\fre_2^{t},\partial_t, \fre_3^{t})=\frac{(k b^2 c+a b')(k b c^2+a c')(c b'-b c')}{(abc(-3k b^2 c^2-2 a c b'+ab c'))}\,.
\end{equation}
For the latter to be zero, either one of the three terms in brackets must be zero. However, if any of the two first terms in brackets is zero, then $\Lambda=0$, so we discard this possibility and hence $c b'-bc'=0$, which in turn implies that $c(t)=b(t)$ since $c(t_0)=b(t_0)=1$. Imposing then that $c(t)=b(t)$ we find that all the components of $\mathcal{W}^g$ vanish identically and, collecting all the results derived up to this point, we arrive to \eqref{eq:edoskspace}. We check as well that equations \eqref{eq:edoskspace} are consistent with \eqref{eq:secondaspace},\eqref{eq:secondbspace} and \eqref{eq:secondcspace}.  Finally, we observe that $h(t)=0$ and $c(t)=b(t)$ is equivalent to the metric adopting the form $g=-\d t^2- b^{-2}(t) (e_{t_0}^1 \otimes e_{t_0}^1- e_{t_0}^2 \otimes e_{t_0}^2)+a^{-2}(t) e_{t_0}^3 \otimes e_{t_0}^3$, so it is clear that the metric has a singularity whenever any of the functions $a(t)$ or $b(t)$  converges to zero or diverges. Hence $\mathcal{I}'_s$, the interval of definition of the ansatz \eqref{eq:anssiem2}, must coincide with $\mathcal{I}$ and we conclude.
\end{proof}

\begin{prop}
\label{prop:spaceconfk}
Let $(M,g)$ be a spacelike quaternionic paraK\"ahler Heisenberg four-manifold. Then:
\begin{itemize}
\item The eigenvalues of the Weyl tensor, understood as a symmetric endomorphism of the bundle of two-forms, are given by $(-2\nu', \nu' , \nu')$, where:
\begin{equation*}
\nu'=\frac{16 k^3 b^7}{\varepsilon a^3(3k b^3+a b')}\,.
\end{equation*}
\item $(M,g)$ is conformally paraK\"ahler for two paracomplex structures with opposite orientations.
\end{itemize}
\end{prop}
\begin{proof}
The proof is analogous to the timelike case. We define:
\begin{equation}\label{SDP:eq}
\omega_i= \d t \wedge \fre_t^i+ \star (\d t \wedge \fre_t^i)\, .
\end{equation}
Interpreting the Weyl tensor as a symmetric endomorphism of the bundle of two-forms in the canonical way, we have:
\begin{equation*}
\mathrm{W}^g(\omega_1)=\nu' \omega_1 \, , \quad \mathrm{W}^g(\omega_2)= \nu' \omega_2\, ,\quad \mathrm{W}^g(\omega_3)= -2\nu' \omega_3\,.
\end{equation*}
This proves the first bullet-point of the Proposition. With regard to the second one, observe that the rescaled two-form
\begin{equation*}
\tilde{\omega}_3=\vert \mathrm{W}^g \vert_g^{2/3} \omega_3
\end{equation*}
is closed and satisfies that $\vert \tilde{\omega}_3 \vert_{\tilde{g}}^2=-4$ with respect to the rescaled metric $\tilde{g}=\vert \mathrm{W}^g \vert_g^{2/3} g$. Then, using that the following two-forms
\begin{equation*}
\tilde{\omega}_2= ab \, \omega_2\, , \quad \tilde{\omega}_3= ab\, \omega_3\,
\end{equation*}
are closed, direct application of the para-version of Lemma \ref{int:lem}, see Remark~\ref{remarktolemma}, proves that $\tilde{g}$ is paraK\"ahler with paraK\"ahler form $\tilde{\omega}_3$ and we conclude. A conformally paraK\"ahler structure for the opposite orientation can be obtained by introducing a relative sign in (\ref{SDP:eq}). 
\end{proof}

\begin{rem}
In complete analogy to the situation in the timelike case, in the proof of Proposition \ref{prop:spaceconfk} we have shown that $(M,g)$ not only is conformally paraK\"ahler but admits as well two almost (para)K\"ahler structures $\tilde{\omega}_1$ and $\tilde{\omega}_2$ compatible with a second conformally rescaled metric $g':= ab \, g$, such that $\vert \tilde{\omega}_1 \vert_{g'}^2=-\vert \tilde{\omega}_2 \vert_{g'}^2=4$.
\end{rem}
We can construct all solutions to equations \eqref{eq:edoskspace} from solutions to the equations \eqref{eq:edosqk} for timelike quaternionic paraK\"ahler Heisenberg four-manifolds, as the following Proposition shows.
\begin{prop}
\label{prop:truco}
Let $(a(t),b(t))$ solve equations \eqref{eq:edosqk} for $\varepsilon=-1$ and a given value of $\Lambda$. Then $(a(2t_0-t),b(2t_0-t))$ solve \eqref{eq:edoskspace} for the same value of $\Lambda$.
\end{prop}
\begin{proof}
Defining $a_{\mathrm{S}}(t)=a(2t_0-t)$ and $b_{\mathrm{S}}(t)=b(2t_0-t)$, we observe that $a_{\mathrm{S}}'(t)=-a'(2t_0-t)$ and $b_{\mathrm{S}}'(t)=-b'(2t_0-t)$. If $a(t)$ and $b(t)$ solve \eqref{eq:edosqk} for a given $\Lambda$, then $a_{\mathrm{S}}(t)$ and $b_{\mathrm{S}}(t)$ solve \eqref{eq:edoskspace}. Furthermore, $a_{\mathrm{S}}(t_0)=a(t_0)=b_{\mathrm{S}}(t_0)=b(t_0)=1$ and we conclude.
\end{proof}
\begin{rem}
If $\mathcal{I}^{\mathrm{timelike}}$ denotes the interval in which a timelike solution $(a(t),b(t))$ is defined, then the spacelike counterpart $(a(2t_0-t),b(2t_0-t))$ is defined in the interval $\mathcal{I}^{\mathrm{spacelike}}=\{t \in \mathbb{R}\, | \, 2t_0-t \in \mathcal{I}^{\mathrm{timelike}}\}$.
\end{rem}
Given this correspondence between timelike and spacelike quaternionic paraK\"ahler Heisenberg four-manifolds, it is natural to split the study into stationary, negative and positive spacelike quaternionic paraK\"ahler Heisenberg four-manifolds, which are obtained from the associated stationary, negative and positive timelike counterparts. 

\subsubsection{Stationary spacelike quaternionic paraK\"ahler Heisenberg four-manifolds}

\begin{mydef}
\label{def:stationaryspace}
A spacelike  quaternionic (para)K\"ahler Heisenberg four-manifold is said to be stationary if $\Lambda=6k^2$.
\end{mydef}

 As in the timelike case, the name stationary comes from the fact that if we set $b=c$ and $h=0$ in the expression for the Ricci tensor in Proposition \ref{prop:riccispace}:
\begin{eqnarray}
\ric^g(\partial_t,\partial_t) &=& -2\lambda^2-\mu^2+2\lambda'+\mu'\, ,\\
\ric^g(\fre_1^t,\fre_1^t) &=& -\frac{2 k^2 b^4}{a^2}-2\lambda^2 -\mu \lambda+\lambda' \, , \\
\ric^g( \fre_2^t, \fre_2^t) &=& \frac{2 k^2 b^4}{a^2}+2\lambda^2+\lambda\mu-\lambda'\, , \\
\ric^g(\fre_3^t,\fre_3^t)&=&-\frac{2 k^2 b^4}{a^2}+2\lambda \mu+\mu^2-\mu'\,,
\end{eqnarray}
where $\mu=(\log a(t) )'$ and  $\lambda=(\log b(t))'$, then if we set $\mu'=\lambda'=0$ we have that $\Lambda=6 k^2$ and hence the name stationary. 

\begin{prop}
\label{prop:statspace}
All stationary spacelike quaternionic paraK\"ahler Heisenberg four-manifolds are given by :
\begin{equation}
a(t)=e^{2k(t-t_0)}\,, \quad b(t)=e^{k(t-t_0)} \, .
\end{equation}
These solutions are isometric to an open orbit of a four-dimensional solvable subgroup (which contains a Heisenberg subgroup) of $\mathrm{SL}(3,\mathbb{R})$ on the symmetric space:
\begin{equation}
\frac{\mathrm{SL}(3,\mathbb{R})}{\mathrm{S}(\mathrm{GL}(1,\mathbb{R})\times \mathrm{GL}(2,\mathbb{R}))}\,.
\label{eq:spacelike}
\end{equation}
Furthermore, they all are incomplete.
\end{prop}
\begin{proof}
These stationary solutions are just obtained by using Proposition \ref{prop:truco} and carrying out the change $t \rightarrow 2t_0-t$ in the stationary solutions found in \eqref{eq:statsol}.  Also, the fact that these solutions are the unique stationary spacelike quaternionic paraK\"ahler Heisenberg four-manifolds follows by the use of Propositions \ref{prop:truco} and \ref{prop:statsol}. 

In complete analogy with the timelike case, we encounter that $\nabla^g \mathrm{R}^g=0$, where $\mathrm{R}^g$ is Riemann curvature tensor of $g$.  Upon use of the classification of pseudo-Riemannian symmetric spaces of quaternionic paraK\"ahler type \cite{AC,Krahe}, we conclude (by comparison of curvature tensors) that the resulting space is locally isometric to the symmetric space \eqref{eq:spacelike}. In particular, the solutions turn out to be isometric to a left-invariant metric on a simply transitive solvable subgroup of $\mathrm{SL}(3,\mathbb{R})$. This can be seen by following analogous arguments to that of the proof of Proposition \ref{prop:statsol}. Finally, we infer that the underlying pseudo-Riemannian manifold is incomplete since the timelike quaternionic paraK\"ahler solution, from which the spacelike one is obtained, is incomplete (and this property does not change under a change of coordinate $t \rightarrow 2t_0-t$). 
\end{proof}
\textbf{Proof of Theorem \ref{thm:introstat}}.   It follows directly upon consideration of Propositions \ref{prop:statsol} and \ref{prop:statspace}. 
\subsubsection{Negative and positive spacelike quaternionic paraK\"ahler Heisenberg four-manifolds}

\begin{mydef}
A spacelike quaternionic paraK\"ahler Heisenberg four-manifold is said to be negative if $\Lambda <0$. Similarly, it is said to be positive if $\Lambda >0$ but $\Lambda \neq 6 k^2$.
\end{mydef}
By the use of the Proposition \ref{prop:truco} together with the classification of timelike quaternionic paraK\"ahler Heisenberg four-manifolds, we may actually obtain all negative and positive spacelike counterparts.

\begin{prop}\label{neg_space_q_para}
Let $(a(t),b(t))$ as in \eqref{eq:uhmpos}. Then $(a(2t_0-t), b(2t_0-t))$ are all solutions to \eqref{eq:edoskspace} with $\Lambda <0$ and, consequently, all the negative spacelike quaternionic paraK\"ahler Heisenberg four-manifolds. They all are incomplete.
\end{prop}

\begin{prop}\label{pos_space_q_para}
Let $(a_1(t),b_1(t))$, $(a_2(t),b_2(t))$ and $(a_3(t),b_3(t))$ be as in \eqref{eq:uhm}, \eqref{eq:uhm2} and \eqref{eq:uhm3}, respectively. Then $(a_1(2t_0-t),b_1(2t_0-t))$ (defined for all $\Lambda >0$), $(a_2(2t_0-t),b_2(2t_0-t))$ and $(a_3(2t_0-t),b_3(2t_0-t))$ (the last two defined for all $\Lambda \geq 81 k^2$) are all solutions to \eqref{eq:edoskspace} with $\Lambda >0$ and $\Lambda \neq 6 k^2$ and, consequently, all the positive spacelike quaternionic paraK\"ahler Heisenberg four-manifolds. Solutions $(a_2(2t_0-t),b_2(2t_0-t))$ and $(a_3(2t_0-t),b_3(2t_0-t))$,  as well as those arising from $(a_1(2t_0-t),b_1(2t_0-t))$ for $\Lambda<6k^2$, are incomplete\footnote{We expect, in the light of Remark \ref{rem:incomplequizas}, solutions $(a_1(2t_0-t),b_1(2t_0-t))$ with $\Lambda>6k^2$ to be incomplete too.}.
 \end{prop}
\begin{proof}
The previous two Propositions are shown by direct use of Proposition \ref{prop:truco} and by the fact that all timelike quaternionic paraK\"ahler Heisenberg four-manifolds are incomplete, a geometric property that does not change after the trivial change of coordinate $t \rightarrow 2t_0-t$.
\end{proof}
\begin{rem}
Note that negative timelike quaternionic paraK\"ahler Heisenberg four-manifolds are in correspondence with their positive spacelike counterparts and vice versa. However, we must bear in mind that in the negative timelike case $\Lambda>0$ while $\Lambda <0$ for positive timelike quaternionic paraK\"ahler Heisenberg four-manifolds, and thus the results are consistent. 
\end{rem}
\begin{rem}
Positive spacelike quaternionic paraK\"ahler Heisenberg four-manifolds correspond as well to neutral-signature analogues of the one-loop deformed universal hypermultiplet, since they can be obtained from the so-called Euclidean supergravity $c$-map \cite{CMMS2}. In fact, if in Equation (59) of Reference \cite{DV} we choose $\bar{M}$ to be a point, $\epsilon_1=1$, $\epsilon_2=\pm 1$, $I=0$, $z^0=1$ and $F_0=\frac{i_{\epsilon_1}}{2} (X^0)^2$, $e^{\mathcal{K}}=-1/2$ and $(\hat{H})^{ab}=\mathrm{diag}(1,-1)$, we observe that after an equivalent procedure to that of \ref{rem:reluhmvc} we have, up to a global sign, the positive spacelike quaternionic paraK\"ahler Heisenberg four-manifolds derived before. Regarding their negative counterparts, we may interpret them as negatively-curved versions of the neutral-signature one-loop deformed universal hypermultiplet.
\end{rem}
\textbf{Proof of Theorem \ref{thm:alllspaceqknonstat}}.  Just by taking into account Propositions \ref{prop:truco}, \ref{neg_space_q_para} and \ref{pos_space_q_para}. 

\subsection{Lightlike quaternionic paraK\"ahler Heisenberg four-manifolds}

Finally we classify all lightlike quaternionic paraK\"ahler Heisenberg four-manifolds. We will make use of the ansatz \eqref{eq:anslightsim}, which gives us a simple way to describe, through a suitable Witt basis $\{\fre_u^t, \fre_v^t, \fre_3^t\}$, the corresponding metrics of lightlike quaternionic paraK\"ahler Heisenberg four-manifolds. We rewrite here this ansatz,  valid in principle in a subinterval $\mathcal{I}'_l \subset \mathcal{I}$:
\begin{equation}
U_W^t=\begin{pmatrix}
1 & 0& 0\\
f(t) & b(t) & p(t) \\
0 &0 & a(t)\\ \end{pmatrix}\,, \quad [\fre_v^{t_0}, \fre_3^{t_0}]=-2 k \fre_u^{t_0}\,, \quad k > 0 \,,
\label{eq:anslightsim2}
\end{equation}
We write first the Ricci curvature tensor $\ric^g$ arising from \eqref{eq:anslightsim2}. 
\begin{prop}
\label{prop:riccilight}
Let $(M,g)$ be a lightlike quaternionic paraK\"ahler Heisenberg four-manifold. The non-zero components of the Ricci curvature tensor $\ric^g$ read:
\begin{eqnarray}
\label{eq:riccitt}
\ric^g(\p_t, \p_t)&=&\frac{a a''-2 (a')^2}{a^2}-\frac{3 (b')^2-2 b b''}{2 b^2}\,, \\ \label{eq:riccilighteuev}\ric^g(\fre_u^t, \fre_v^t)&=& \frac{(b')^2}{b^2}+ \frac{a'b'-a b''}{2a b }\, , \\ \label{eq:riccilighte2}
\ric^g(\fre_v^t, \fre_v^t) & =&  \frac{a' (b f'-f b')}{ab}+\frac{(p b'-bp')^2}{2a^2 b^2}-\frac{3 f (b')^2-b(3  b'f'+f b''))}{b^2 } - \frac{f''}{a^2}\,,\\
 \label{eq:riccilighte2e3}
\ric^g(\fre_v^t, \fre_3^t)&=&\frac{-2 a p (b')^2+b(2 a b' p'+p(-3 a' b'+a b''))+b^2(3 a' p'-a p'')}{2a^2 b^2}\,, \\
\label{eq:riccilighte3}
\ric^g(\fre_3^t, \fre_3^t) &=&\frac{2 (a')^2 -a a''}{a^2}+\frac{a'b'}{a b}\,.
\end{eqnarray}
\end{prop}
\begin{proof}
By direct computation. 
\end{proof}
\begin{prop}
\label{prop:lightcase}
Let $(M,g)$ be a lightlike Heisenberg four-manifold. It is Einstein with Einstein constant $\Lambda \neq 0$ if and only if:
\begin{eqnarray*}
a''&=&-\Lambda  a+\frac{2(a')^2}{a}+a'\frac{b'}{b} \,,\\
b'' &=& -\Lambda b+\frac{7 (b')^2}{4 b}\,, \\
f''&=&\frac{a'(-f b'+b f')}{a b}+\frac{(p b'-b p')^2}{2 a^2 b^2}+\frac{(-3 f (b')^2+ 3 b b' f'+b f b''))}{b^2 }\, , \\
p'' &=& \frac{-2 p (b')^2}{b^2}+\frac{3 a' p'}{a}+\frac{2 a b'p' + p(-3 a' b'+a b'')}{a b}\,,\\
a'&=&a \left (\frac{\Lambda b}{b'}-\frac{b'}{4b} \right)\,.
\end{eqnarray*}
 \end{prop}
\begin{proof}
Assume $\ric^g=\Lambda g$. Using Proposition \ref{prop:riccilight}, from \eqref{eq:riccilighte3} we may solve for $a''$:
\begin{equation}
\label{eq:secondalight}
a''=-\Lambda  a+\frac{2(a')^2}{a}+a' \frac{b'}{b}\,.
\end{equation}
Now from \eqref{eq:riccilighte2} and  \eqref{eq:riccilighte2e3} we may solve for $f''$ and $p''$:
\begin{eqnarray*}
f''&=& \frac{a'(-f b'+b f')}{a b}+\frac{(p b'-b p')^2}{2 a^2 b^2}+\frac{(-3 f (b')^2+ 3 b b' f'+b f b''))}{b^2 }\, , \\
p'' &=& \frac{-2 p (b')^2}{b^2}+\frac{3 a' p'}{a}+\frac{2 a b'p' + p(-3 a' b'+a b'')}{a b}\,.
\end{eqnarray*}
Substituting these results we have obtained on the rest of the components of the Ricci tensor, the Einstein condition is fulfilled if:
\begin{equation*}
a'=a \left (\frac{\Lambda b}{b'}-\frac{b'}{4b} \right)\,, \quad b''=-\Lambda b+\frac{7 (b')^2}{4 b}\,.
\end{equation*}
After imposing this last condition, we observe that $\ric^g =\Lambda g$. Checking that the derivative of previous equation is consistent with \eqref{eq:secondalight}, we collect now all the results we have obtained and we conclude. 
\end{proof}
In order to classify all lightlike quaternionic paraK\"ahler Heisenberg four-manifolds, we need to know the conditions for an Einstein lightlike Heisenberg four-manifold to be half-conformally flat, which is equivalent to $\mathcal{W}^g=0$.

\begin{prop}
Let $(M,g)$ be a lightlike Einstein Heisenberg four-manifold with Einstein constant $\Lambda \neq 0$. The Weyl self-duality tensor $\mathcal{W}^g$ vanishes identically if and only if:
\begin{equation*}
b(t)=e^{\frac{2 \sqrt{\Lambda}}{\sqrt{3}} (t-t_0) }\, , \quad a(t)=e^{\sqrt{\frac{\Lambda}{3}}(t-t_0)}\, , \quad f(t)=p(t)=0\,, \quad \Lambda>0\,.
\end{equation*}
In particular, if $(M,g)$ is half-conformally flat, it is actually conformally flat.
\end{prop}
\begin{proof}
Using the results of Proposition \ref{prop:lightcase}, we find:
\begin{equation*}
\mathcal{W}(\fre_u^t,\partial_t,\partial_t, \fre_v^t)=\frac{\Lambda}{6} -\frac{1}{8}\left ( \frac{b'}{b} \right)^2\,.
\end{equation*}
Hence:
\begin{equation}
\label{eq:lambdalight}
\Lambda=\frac{3}{4}\left (\frac{b'}{b} \right )^2>0\,.
\end{equation}
Note that this equation is consistent, since by differentiating with respect to $t$ and using the expression for $b''$ found in Proposition \ref{prop:lightcase}, we indeed get that the right-hand side of the previous equation is indeed zero. Similarly, we have that:
\begin{equation*}
\mathcal{W}(\partial_t, \fre_3^t, \fre_v^t, \fre_3^t)=\frac{b'(-p b'+b p')}{4 a b^2 }=0\,.
\end{equation*}
If $b'=0$, this would imply in turn that $\Lambda=0$. Since we are assuming that $\Lambda \neq 0$, we find that $-p b'+b p'=0$, which is equivalent to $p(t)=0$ on taking into account that $p(t_0)=0$ and that $b(t) \neq 0$ since otherwise $U^t_W$ would be degenerate. Finally, on substituting these results, we also have that: 
\begin{equation*}
\mathcal{W}( \fre_v^t,\partial_t, \fre_v^t, \partial_t)=\frac{ b'(-f b'+b f')}{2 b^2}=0\,.
\end{equation*}
From here we find that $f(t)=0$ and then we encounter that not only the self-duality tensor vanishes, but also the Weyl tensor itself, so the subsequent metric is conformally flat. Simplifying the result for $a'$ obtained in Proposition \ref{prop:lightcase} by using \eqref{eq:lambdalight}, we have that the remaining differential equations to solve are: 
\begin{equation*}
a'=a\frac{b' }{2b}\, , \quad \Lambda=\frac{3}{4}\left (\frac{b'}{b}\right )^2\,.
\end{equation*}
The solution to the previous system of ordinary ODEs with the initial conditions $a(t_0)=b(t_0)=1$ is:
\begin{equation*}
b(t)= e^{\frac{2 \sqrt{\Lambda}}{\sqrt{3}} (t-t_0) }\, , \quad a(t)=e^{\sqrt{\frac{\Lambda}{3}}(t-t_0)}\, ,
\end{equation*}
and we conclude. 
\end{proof}
\begin{rem}
The metric $g$, in terms of the coframe $\{\d t, \fre^1_{t_0},\fre^2_{t_0},\fre^3_{t_0}\}$ reads:
\begin{equation*}
\begin{split}
g&=-\d t^2 +\frac{1}{b(t)} \fre^u_{t_0} \odot \fre^v_{t_0}+\frac{1}{(a(t))^2}\fre^3_{t_0} \otimes \fre^3_{t_0}\\&=-\d t^2+e^{-\frac{2 \sqrt{\Lambda}}{\sqrt{3}} (t-t_0) }\fre^u_{t_0} \odot \fre^v_{t_0}+e^{-\frac{2 \sqrt{\Lambda}}{\sqrt{3}} (t-t_0) } \fre^3_{t_0} \otimes \fre^3_{t_0}\,.
\end{split}
\end{equation*}
We observe from the previous expression that the metric is indeed conformally flat: by defining a new coordinate $\d t=e^{-\frac{ \sqrt{\Lambda}}{\sqrt{3}} (t(\tilde{t})-t_0) } \d \tilde{t}$, it is clear that the metric is conformally flat with conformal factor $e^{-\frac{2 \sqrt{\Lambda}}{\sqrt{3}}( t(\tilde{t})-t_0) }$.
\end{rem}
\begin{prop}\label{prop:light}
All lightlike quaternionic paraK\"ahler Heisenberg four-manifolds are conformally flat and isometric to the solution given by:
\begin{equation}
b(t)=e^{\frac{2 \sqrt{\Lambda}}{\sqrt{3}} (t-t_0) }\, , \quad a(t)=e^{\sqrt{\frac{\Lambda}{3}}(t-t_0)}\, , \quad f(t)=p(t)=0\,, \quad \Lambda>0\,,
\label{eq:solproplight}
\end{equation}
where $t \in \mathcal{I}'_l=\mathcal{I}=\mathbb{R}$. The subsequent pseudo-Riemannian manifolds are incomplete.
\end{prop}
\begin{proof}
The fact that $t \in \mathcal{I}'_l=\mathcal{I}=\mathbb{R}$ follows by seeing that \eqref{eq:solproplight} is defined in the entire real line. Consequently, we only have to show that the corresponding metrics are incomplete. For that, let us set $t_0=0$ for the sake of simplicity (we can always achieve it by shifting the time coordinate) and let us use the coordinates \eqref{eq:coordh} for $\mathrm{H}$. We consider the geodesic $\Gamma: J \rightarrow (\mathcal{I} \times \mathrm{H})$ with $J \subset \mathbb{R}$ whose coordinates (which we assume to be affinely-parametrized by $\tau$) are given by:
\begin{equation}
\left ( -\frac{\sqrt{3}}{2\sqrt{\Lambda}} \log \left (  \frac{4 \Lambda \sinh^2( \tau/2+B)}{3}\right),0, \int_{0}^\tau \frac{3}{4 \Lambda \sinh^2(\sigma/2+B)} \d \sigma,0 \right)\,,
\end{equation}
where we choose a certain $B \in \mathbb{R}^{>0}$. This geodesic is not defined $\forall \tau \in \mathbb{R}$ and  we conclude that the underlying pseudo-Riemannian manifold cannot be complete. 
\end{proof}
\begin{rem}
\label{rem:lorsollight}
Differently from what happens in the previous cases, when the Heisenberg center is lightlike it is possible to find Lorentzian conformally flat\footnote{
Note that for a Lorentzian metric the two components of the Weyl tensor ($\pm i$-eigenvectors of the Hodge operator) are related by complex conjugation. 
So half-conformal flatness in the Lorentzian setting implies conformal flatness.} Einstein metrics. 
Setting the metric to be:
\begin{equation}
g^{\mathrm{Lor}}=\d t^2 + \fre_t^{u} \odot \fre_t^{v} + \fre_t^3 \otimes \fre_t^3\,,
\end{equation}
we can find, in a completely analogous manner to that presented in the study of lightlike quaternionic paraK\"ahler Heisenberg four-manifolds, that the following choice for the functions in $U_W^t$:
\begin{equation}
b(t)=e^{\frac{2 \sqrt{-\Lambda}}{\sqrt{3}} (t-t_0) }\, , \quad a(t)=e^{\sqrt{\frac{-\Lambda}{3}}(t-t_0)}\, ,\quad f(t)=p(t)=0\,, \quad \Lambda<0\,
\end{equation}
yields Lorentzian conformally flat Einstein metrics. As for  the neutral-signature metrics, these metrics are incomplete. 
\end{rem}
\textbf{Proof of Theorem \ref{thm:alllightqknonstat}}.  It follows on noticing that Theorem  \ref{thm:alllightqknonstat} is equivalent to Proposition \ref{prop:light}.
\section{(Para)HyperK\"ahler Heisenberg four-manifolds}

\label{sec:4}


In this section we are going to classify all Heisenberg four-manifolds which are furthermore (para)hyperK\"ahler. For that, we begin by providing the most adequate definition of (para)hyperK\"ahler four-manifold  for our purposes.
\begin{mydef}
Let $(M,g)$ be a neutral-signature or Riemannian orientable four-manifold. It is said to be (para)hyperK\"ahler if there exist three closed self-dual two-forms (called K\"ahler forms) $\omega_i$ on $M$, with $i=1,2,3$ or $i=u,v,3$ which satisfy the condition:
\begin{equation}
\omega_i \wedge \omega_j=2 \varepsilon \eta_{i j}\mathrm{dvol}_M\, , \quad i,j=1,2,3\,, \quad \mathrm{or} \quad i,j=u,v,3\,,
\end{equation}
where $\varepsilon=1$ if $(M,g)$ is Riemannian and $\varepsilon=-1$ if $(M,g)$ is of  neutral signature, $\eta$ is given by \eqref{eq:etaexpression} and where $\mathrm{dvol}_M$ denotes the canonical volume form given by the metric and the fixed orientation on $(M,g)$.
\end{mydef}
\begin{rem}
\label{rem:parallelhk}
The usual definition of (para)hyperK\"ahler four-manifold $(M,g)$ is that of a  Riemannian (resp. neutral-signature) four-manifold whose holonomy group is contained in the symplectic group $\mathrm{Sp}(1)$ (resp. in the pseudo-symplectic group $\mathrm{SL}(2,\mathbb{R})$). This is equivalent to the existence of three parallel (with respect to the Levi-Civita connection of $(M,g)$) (para)complex structures $J_i \in \mathrm{End}(TM)$ which  are antisymmetric and satisfy\footnote{We allow the possibility of having two (integrable) nilpotent endomorphisms and a paracomplex structure. Note however that these can be obtained from linear combinations of two paracomplex structures and a complex one satisfying the relations $J_i \circ J_j+J_j \circ J_i=-\varepsilon \eta_{ij} \mathrm{Id}_{TM}$.} $J_i \circ J_j+J_j \circ J_i=-\varepsilon \eta_{ij} \mathrm{Id}_{TM}$. The relation between this definition and ours is given by the Hitchin lemma \cite{H}, which can be rephrased by saying that the existence of three closed and self-dual two-forms $\omega_i \in \Omega^2(M)$ on $(M,g)$ satisfying $\omega_i \wedge \omega_j=2 \varepsilon \eta_{i j}\mathrm{dvol}_M$ is equivalent to the existence of three (integrable) parallel (para)complex structures $J_i$ satisfying the  aforementioned relations. 
\end{rem}
\begin{rem}
We would like to remind the reader that (para)hyperK\"ahler manifolds are Ricci-flat. 
\end{rem}
From the previous Remark and Proposition \ref{prop:zcausalconst}, we have that timelike, spacelike and lightlike parahyperK\"ahler Heisenberg four-manifolds comprise all possible types of parahyperK\"ahler Heisenberg four-manifolds. We introduce the following notation.
\begin{mydef}
Let $(M,g)$ be a Riemannian or neutral-signature Heisenberg four-manifold and let $\{\mathfrak{e}_i^t\}_{t \in \mathcal{I}}$ be a family of orthonormal or Witt frames on $\mathrm{H}$. For $t_0 \in \mathcal{I}$, we say $\{\mathfrak{e}_i^t\}_{t \in \mathcal{I}}$ is a $t_0$-canonical frame if the only non-vanishing Lie bracket at $t_0$ is:
\begin{itemize}
\item For $(M,g)$ a Riemannian or a timelike Heisenberg four-manifold:
\begin{equation}
\label{eq:hkcanlbriem}
[\fre_2^{t_0}, \fre_3^{t_0}]=-2 k \fre_1^{t_0}\, , \quad k >0\,.
\end{equation}
\item For spacelike Heisenberg four-manifolds:
\begin{equation}
\label{eq:hkcanlbspace}
[\fre_1^{t_0}, \fre_2^{t_0}]=-2 k \fre_3^{t_0}\, , \quad k >0\,.
\end{equation}
\item For lightlike Heisenberg four-manifolds:
\begin{equation}
\label{eq:hkcanlblight}
[\fre_v^{t_0}, \fre_3^{t_0}]=-2 k \fre_u^{t_0}\, , \quad k >0\,.
\end{equation}
\end{itemize}
\end{mydef}
\begin{lem}
\label{lemma:hkres}
Let $(M,g)$ be a Riemannian or neutral-signature Heisenberg four-manifold.  It is (para)hyperK\"ahler if and only if, for $t_0 \in \mathcal{I}$, there exists a $t_0$-canonical frame $\{\mathfrak{e}_i^t\}_{t \in \mathcal{I}}$ such that the following self-dual two-forms:
\begin{equation}
\omega_i= \sigma_{ij} \d t \wedge \mathfrak{e}_t^j+ \sigma_{ij}  \star (\d t \wedge \mathfrak{e}_t^j)\, .
	\label{eq:hkforms}
\end{equation}
are closed, where $\{\mathfrak{e}^i_t\}_{t \in \mathcal{I}}$ denotes the associated family of coframes dual to the $t_0$-canonical frame and $\sigma_{i_1j}\in C^\infty(\mathcal{I})$, $\sigma_{i_2j}, \sigma_{i_3j} \in C^\infty(M)$ such that $(\sigma_{ij}(p))\in \mathrm{SO}(\mathbb{R}^3,\eta )$ for all $p\in M$ with:
\begin{itemize}
\item $i_1=1$, $i_2=2$ and $i_3=3$ if $(M,g)$ is a Riemannian or a timelike Heisenberg four-manifold,
\item $i_1=3$, $i_2=1$ and $i_3=2$ if $(M,g)$ is a spacelike Heisenberg four-manifold,
\item $i_1=u$, $i_2=v$ and $i_3=3$ if $(M,g)$ is a  lightlike Heisenberg four-manifold.
\end{itemize}
\end{lem}
\begin{proof}
Assume first that the self-dual two-forms $\omega_i$  in Equation (\ref{eq:hkforms}) are closed. By direct computation, we check that $\omega_i \wedge \omega_j= 2 \varepsilon \eta_{ij} \mathrm{dvol}_M$ and, therefore, $(M,g)$ is (para)hyperK\"ahler. 

Conversely, assume that $(M,g)$ is a (para)hyperK\"ahler Heisenberg four-manifold. Let $\omega_i$ be the corresponding  K\"ahler forms. Note first that the (para)quaternionic structure $Q_p$ at any  point $p\in M$ is defined by one of the two simple ideals of $\mathfrak{so}(T_pM) = Q_p \oplus Q'_p$, where $Q_p\cong Q_p'\cong \mathfrak{su}(2)$ or $Q_p\cong Q_p'\cong \mathfrak{sl}(2,\mathbb{R})$. Therefore it is invariant under any orientation-preserving isometry. Consider the vector space $V\cong\mathbb{R}^3$ consisting of all parallel sections of $Q$ 
endowed with the Euclidean or Lorentzian scalar product $\langle A,B\rangle = -\frac{\varepsilon}{4} \tr AB$. Since the Heisenberg group acts through orientation-preserving isometries, we obtain a representation $\rho:\mathrm{H} \to \mathrm{SO}(V)$, whose image is a nilpotent subgroup of $\mathrm{SO}(V)$. In the Riemannian case, this leads us to the conclusion that the image of $\rho$ is contained in an $\mathrm{SO}(2)$-subgroup and, therefore, preserves a non-zero vector in $V$. In the neutral-signature case, the image is contained in a one-dimensional subgroup  conjugate to $\mathrm{SO}(2)$, 
$\mathrm{SO}_0(1,1)$ or to a unipotent group that preserves a lightlike vector.  Again, we can conclude that the representation $\rho$, independently of the signature, always leaves invariant a non-zero vector of $V$.
This implies there is an orthonormal or Witt basis $(J_1,J_2,J_3)$ of $V$ which contains an invariant 
element. (We recall  that according to our conventions $\langle J_i,J_j\rangle = \eta_{ij}$ with $\eta$ as in (\ref{eq:etaexpression}).) If the element has stabilizer $\mathrm{SO}(2)$ in $\mathrm{SO}(3)$, we can assume it is $J_1$. If it has stabilizer $\mathrm{SO}(1,1)$ we can take it to be $J_3$. If it is lightlike, the basis $(J_1,J_2,J_3)$ is a Witt basis 
$(J_u,J_v,J_3)$ and we can assume that the invariant element is $J_u$. 
Then the corresponding 
left-invariant K\"ahler forms $\omega_{i_1}$ can be chosen so that $\sigma_{i_1 j} \in C^\infty(\mathcal{I})$ with $i_1 \in \{1,3,u\}$ as in the the statement of the lemma and we conclude.

\end{proof}

After these preliminary results, now we continue with the classification of Riemannian and neutral-signature (para)hyperK\"ahler Heisenberg four-manifolds.

\subsection{HyperK\"ahler and timelike parahyperK\"ahler Heisenberg four-manifolds}

We carry out the classification of hyperK\"ahler and  timelike parahyperK\"ahler Heisenberg four-manifolds at once, since we will see that the procedure is strictly analogous. We fix a $t_0$-canonical frame $\{\mathfrak{e}_i^t\}_{t \in \mathcal{I}}$ which satisfies \eqref{eq:hkcanlbriem} and we set the ansatz \eqref{eq:ansriemsim} for the matrix $U^t$:
\begin{equation}
U^t=\begin{pmatrix}
a(t) &0& 0\\
0 & b(t) & 0 \\
0 & h(t) & c(t) \\
\end{pmatrix}\,,  \quad a,b,c,h \in C^\infty(\mathcal{I})\,.
\label{eq:ansuhkriem}
\end{equation}


\begin{prop} \label{prop:hkriemsol}
Let $(M,g)$ be a (timelike) (para)hyperK\"ahler Heisenberg four-manifold. Then it is isometric to:
\begin{equation}
\label{eq:propsol1}
a=(1+3 k (t-t_0))^{1/3}\, ,\quad  b=c=(1+3 k (t-t_0))^{-1/3}\, , \quad h=0\,.
\end{equation}
The maximal domain  
of definition of these incomplete metrics is $(t_0-(3k)^{-1}, +\infty) \times \mathrm{H}$.
\end{prop}
\begin{proof}
According to Lemma \ref{lemma:hkres}, there exists (at least) a K\"ahler form $\omega_1$ belonging to the (para)hyperK\"ahler structure which is additionally invariant under the Heisenberg group action. If $\omega_1=\sigma_{1j} \d t \wedge \mathfrak{e}_t^j+\sigma_{1j}  \star (\d t \wedge \mathfrak{e}_t^j) $, then it can be seen that the equation $\nabla \omega_1=0$ is equivalent to:
\begin{equation*}
\sigma'_{11}=0\, , \quad \frac{\sigma_{12}'}{\sigma_{13}}=-\frac{\sigma_{13}'}{\sigma_{12}}=\frac{hc'-h'c}{2bc}\,, \quad a'= k bc\, , \quad h'=\frac{c'}{c} h \, , \quad \frac{b'}{b}=\frac{c'}{c}=-k \frac{bc}{a}\,.
\end{equation*}
The unique solution to the previous system of ordinary differential equations with the initial conditions $a(t_0)=b(t_0)=c(t_0)=1$, $h(t_0)=0$ and $\sigma_{1j} \vert_{t_0}=\sigma_{1j}^0$, for $\sigma_{1j}^0 \in \mathbb{R}$, turns out to be:
\begin{equation}
\label{eq:provehkrt}
a=(1+3 k (t-t_0))^{1/3}\, ,\quad  b=c=(1+3 k (t-t_0))^{-1/3}\, , \quad h=0\, , \quad \sigma_{1j} =\sigma_{1j}^0\,.
\end{equation}
This solution is defined in the interval $\mathcal{I}=(t_0-(3k)^{-1}, +\infty)$ and we observe that the isometry type of $(M,g)$ is completely fixed. Now, using \eqref{eq:provehkrt}, we note that the following two-forms:
\begin{equation*}
\omega_i=\d t \wedge \fre_t^i + \star (\d t \wedge \fre_t^i)\,  , \quad i=1,2,3
\end{equation*}
are self-dual, closed and satisfy $\omega_i \wedge \omega_j= 2 \varepsilon \eta_{ij} \mathrm{dvol}_M$. Hence we conclude.
\end{proof}
\begin{rem}
In terms of the coframe $\{\d t,\mathfrak{e}_{t_0}^i\}$, the metric of a (timelike) (para)hyperK\"ahler Heisenberg four-manifold $(M,g)$ reads:
\begin{equation}
g=\varepsilon \d t^2+ \frac{\varepsilon}{(1+3 k (t-t_0))^{2/3}} \mathfrak{e}_{t_0}^1 \otimes \mathfrak{e}_{t_0}^1+ (1+  3k (t-t_0))^{2/3}\left ( \mathfrak{e}_{t_0}^2 \otimes \mathfrak{e}_{t_0}^2+\mathfrak{e}_{t_0}^3 \otimes \mathfrak{e}_{t_0}^3 \right ) \,.
\label{eq:mhk}
\end{equation}
We find that $(M,g)$ is Ricci-flat and that the Weyl tensor is antiself-dual.
\end{rem}
\begin{rem}\label{conf:Rem}
Redefine the time coordinate as $e^{3 k\tilde{t}}=1+3k (t-t_0)$. Then the metric $g$ reads:
\begin{equation}
g=\varepsilon e^{6 k \tilde{t}} \d \tilde{t}^2+\varepsilon  e^{-2k \tilde{t}}\fre_{t_0}^1 \otimes \fre_{t_0}^1 +e^{2k \tilde{t}} (\fre_{t_0}^2 \otimes \fre_{t_0}^2+\fre_{t_0}^3\otimes \fre_{t_0}^3)\,.
\end{equation}
Now if we consider the rescaled metric $\hat{g}=e^{-6 k \tilde{t}} g$, we observe that:
\begin{equation}
\hat{g}=\d \tilde{t}^2+\fre_{\tilde{t}}^1\otimes\fre_{\tilde{t}}^1+\fre_{\tilde{t}}^2\otimes \fre_{\tilde{t}}^2+\fre_{\tilde{t}}^3\otimes \fre_{\tilde{t}}^3\,,
\end{equation}
with $\fre_{\tilde{t}}^1=e^{-4k \tilde{t}} \fre_{t_0}^1$, $\fre_{\tilde{t}}^2=e^{-2 k\tilde{t}} \fre_{t_0}^2$ and $\fre_{\tilde{t}}^3=e^{-2k\tilde{t}} \fre_{t_0}^3$. We compute:
\begin{equation}
\d \fre_{\tilde{t}}^1=4 k \fre_{\tilde{t}}^1 \wedge \d \tilde{t} +2 k \fre_{\tilde{t}}^2 \wedge \fre_{\tilde{t}}^3\, , \quad \d \fre_{\tilde{t}}^2=2 k \fre_{\tilde{t}}^2 \wedge \d \tilde{t} \, , \quad \d \fre_{\tilde{t}}^3=2k \fre_{\tilde{t}}^3 \wedge \d \tilde{t}\,.
\end{equation}
So we observe that $\{\d t, \fre_{\tilde{t}}^1,\fre_{\tilde{t}}^2 ,\fre_{\tilde{t}}^3\}$ defines a left-invariant coframe on $\mathbb{R}^{+}\times H$. Therefore the singularity in the metric \eqref{eq:mhk} is just present up to a conformal factor, compare with \cite{DH}. 
\end{rem}
\subsection{Spacelike parahyperK\"ahler Heisenberg four-manifolds}

Now we continue with the classification of spacelike parahyperK\"ahler Heisenberg four-manifolds. We set a $t_0$-canonical frame $\{\mathfrak{e}_i^t\}_{t \in \mathcal{I}}$ which satisfies \eqref{eq:hkcanlbspace} and we use the ansatz \eqref{eq:ansspacesim} for the matrix $U^t$ (valid for a subinterval $\mathcal{I}'_s \subset \mathcal{I}$ containing $t_0$), which we rewrite here for the sake of clarity:
\begin{equation}
U^t=\begin{pmatrix}
c(t) &h(t) & 0\\
-h(t) & b(t) & 0 \\
0 & 0 & a(t) \\
\end{pmatrix}\,, \quad  a,b,c,h \in C^\infty(\mathcal{I}'_s)\,.
\label{eq:ansuhkspace}
\end{equation}
\begin{prop} \label{prop:hkspacesol}
Let $(M,g)$ be a spacelike  parahyperK\"ahler Heisenberg four-manifold. Then it is isometric to:
\begin{equation}
\label{eq:propsolspace1}
a=(1-3 k (t-t_0))^{1/3}\, ,\quad  b=c=(1-3 k (t-t_0))^{-1/3}\, , \quad h=0\,.
\end{equation}
The maximal domain  
of definition of these incomplete metrics is $(-\infty,t_0+(3k)^{-1}) \times \mathrm{H}$.
\end{prop}
\begin{proof}
In an analogous fashion to the proof of Proposition \ref{prop:hkriemsol}, we consider a K\"ahler form $\omega_3$  invariant under the Heisenberg group, which we know to exist by virtue of Lemma \ref{lemma:hkres}. Writing $\omega_3=\sigma_{3j} \d t \wedge \fre_t^j + \sigma_{3j} \star (\d t \wedge \fre_t^j)$, the parallel condition $\nabla \omega_3=0$ implies that:
\begin{eqnarray*}
\frac{\sigma_{31}'}{\sigma_{32}}&= &\frac{\sigma_{32}'}{\sigma_{31}}=\frac{h(c'-b')+h'(b-c)}{2(bc+h^2)}\, , \\  h'&= &h\frac{kb^2c^2+2 k bc h^2+k h^4+a c c'}{a(bc+c^2+h^2)}=h \frac{k b^2c^2+k h^4+2 k bc h^2+a b b'}{a(bc+b^2+h^2)}\,. \\
b'&=&\frac{k b^3 c^2+ k c h^4+ k bh^2(2 c^2+h^2)+ kb^2 c(c^2+2h^2) -a h^2 c'}{a(bc+c^2+h^2)}\, , \\  c'&=& \frac{k b^3 c^2+ k c h^4+ k bh^2(2 c^2+h^2)+ kb^2 c(c^2+2h^2) -a h^2 b'}{a(bc+b^2+h^2)}\, , \\ a'&=&-k(bc+h^2)\,.
\end{eqnarray*}
This is a system of first-order ordinary differential equations with the initial condition $a(t_0)=b(t_0)=c(t_0)=1$, $h(t_0)=0$ and $\sigma_{3j} \vert_{t_0}=\sigma_{3j}^0$, for $\sigma_{3j}^0 \in \mathbb{R}$. It can be see to admit a unique solution, which turns out to be:
\begin{equation}
\label{eq:spacehksol}
a=(1-3 k (t-t_0))^{1/3}\, ,\quad  b=c=(1-3 k (t-t_0))^{-1/3}\, , \quad h=0\,, \quad \sigma_{3j}=\sigma_{3j}^0\,.
\end{equation}
This solution is defined in the interval $\mathcal{I}'_s=(-\infty, t_0+(3k)^{-1}))$. Ansatz \eqref{eq:ansuhkspace} was in principle valid only for a subinterval $\mathcal{I}'_s \subset \mathcal{I}$ containing $t_0$, but we note that actually $\mathcal{I}'_s=\mathcal{I}$, since at $t=t_0+(3k)^{-1}$ the metric has a singularity and cannot be extended for larger values of $t$. Using now
\eqref{eq:spacehksol}, we observe that the following two-forms:
\begin{equation*}
\omega_i=\d t \wedge \fre_t^i + \star (\d t \wedge \fre_t^i)\, , \quad i=1,2,3
\end{equation*}
are self-dual, closed and satisfy $\omega_i \wedge \omega_j= -2\eta_{ij} \mathrm{dvol}_M$. Hence we conclude.
\end{proof}
\begin{rem}
We observe that the subsequent pseudo-Riemannian manifold is Ricci-flat and the Weyl tensor is antiself-dual. The metric turns out to be:
\begin{equation*}
g=- \d t^2+ (1-3k (t-t_0))^{2/3}\left ( -\mathfrak{e}_{t_0}^1 \otimes \mathfrak{e}_{t_0}^1+\mathfrak{e}_{t_0}^2 \otimes \mathfrak{e}_{t_0}^2 \right )+ \frac{1}{(1-3 k (t-t_0))^{2/3}} \mathfrak{e}_{t_0}^3 \otimes \mathfrak{e}_{t_0}^3 \,.
\end{equation*}
\end{rem}

\subsection{Lightlike parahyperK\"ahler Heisenberg four-manifolds}

Finally we carry out the classification of all lightlike parahyperK\"ahler Heisenberg four-manifolds. In analogy with the previous cases, we pick a $t_0$-canonical frame $\{\fre_i^t\}_{t \in \mathcal{I}}$ which satisfies \eqref{eq:hkcanlblight} at $t_0 \in \mathcal{I}$ and we choose the ansatz \eqref{eq:anslightsim} for $U_W^t$ (valid for a subinterval $\mathcal{I}'_l \subset \mathcal{I}$ containing $t_0$), which we present here again:
\begin{equation}
U_W^t=\begin{pmatrix}
1 &0 & 0\\
f(t) & b(t) & p(t) \\
0 & 0 & a(t) \\
\end{pmatrix}\,, \quad a,b,c,f,p \in C^\infty(\mathcal{I}'_l)
\label{eq:ansuhklight}
\end{equation}

\begin{prop} \label{prop:hklightsol}
All lightlike parahyperK\"ahler Heisenberg four-manifolds are isometric to $(\mathbb{R}\times \mathrm{H},g)$, where $g$ is the metric \eqref{eq:methinv} constructed from a $t_0$-canonical frame $\{\fre_i^t\}_{t \in \mathbb{R}}$ such that:
\begin{equation}
a=b=1\, ,\quad  f=-2k (t-t_0)\, , \quad p=0\,
\end{equation}
Furthermore, such metric is flat and isometric to $(\mathbb{R}^4, \eta)$ (and therefore, complete).
\end{prop}
\begin{proof}
Following Lemma \ref{lemma:hkres} and its notation, let $\omega_u$ denote the corresponding K\"ahler form that is additionally invariant under the Heisenberg group, which it is guaranteed to exist. If $\omega_u= \sigma_{uj} \d t \wedge \fre_t^j + \sigma_{uj} \star ( \d t \wedge \fre_t^j)$, then $\omega_u$ being parallel implies:
\begin{eqnarray*}
\sigma_{uu}&=&\frac{1}{2 a b}\left( a b' \sigma_{uu} +(bp'-pb')\sigma_{u3}\right) \, , \quad \sigma_{uv}'=-\frac{\sigma_{uv} b'}{2b}\, , \quad \sigma_{u3}=\frac{\sigma_{uv}(pb'-bp')}{2 ab}\, \\
f' & =&-2 k a b+ \frac{\sigma_{uu} + 2 f\sigma_{uv}}{2 b \sigma_{uv}} b' \, , \quad  p'= \frac{p \sigma_{uv} + a \sigma_{u3}}{b \sigma_{uv}} b'= \frac{p b'}{b}\, , \quad a'= b'=0\,.
\end{eqnarray*}
The unique solution to the previous system of ODEs with the initial conditions $a(t_0)=b(t_0)=1$ and $p(t_0)=f(t_0)=0$ turns out to be\footnote{After using $b'=0$ in the rest of equations, the possible divergences arising from the possibility that $\sigma_{uv}=0$ disappear.}:
\begin{equation}
a=b=1\, , \quad f=-2 k (t-t_0)\, , \quad p=0\, , \quad \sigma_{uj}=\sigma_{uj}^0\, ,
\label{eq:solhklightproof}
\end{equation}
where $\sigma_{uj}^0 \in \mathbb{R}$ are constants. This solution is trivially defined for $t \in \mathbb{R}$, so we can actually extend $\mathcal{I}'_l$ to be the entire $\mathcal{I}$ and $\mathcal{I}'_l=\mathcal{I}=\mathbb{R}$. Using now \eqref{eq:solhklightproof}, we observe that the two-forms:
\begin{equation*}
\omega_i= \d t \wedge \fre_t^i+ \star ( \d t \wedge \fre_t^i) \, , \quad i=u,v,3
\end{equation*}
are self-dual, closed and satisfy $\omega_i \wedge \omega_j =2 \varepsilon \eta_{ij} \d \mathrm{vol}_M$. This way we obtain all lightlike  parahyperK\"ahler Heisenberg four-manifolds, which we easily see to be flat. Finally, after using the coordinates \eqref{eq:coordh}, it is possible to see that all geodesics with coordinates $(t(\tau),x(\tau),y(\tau),z(\tau))$ and affine parameter $\tau$ take the form:
\begin{equation}
\begin{split}
t(\tau)=A_1 +A_2 \tau-k A_3^2 \tau^2\, , \quad & x(\tau)=A_4+A_3 \tau\, , \quad y(\tau)=A_5+A_6 \tau-k A_3^2 \tau^2\, , \\ z(\tau)&=A_7+A_8 \tau+ k A_3 \tau^2( k A_3 x(\tau)-2 A_2+A_6)\,,
\end{split}
\end{equation}
with $A_l \in \mathbb{R}$ for $l=1,2,\dots,8$. Since these geodesics are defined $\forall \tau \in \mathbb{R}$ we conclude that the subsequent pseudo-Riemannian manifolds are complete and therefore all lightlike parahyperK\"ahler Heisenberg four-manifolds are isometric to four-dimensional flat space $(\mathbb{R}^4,\eta)$.

\end{proof}
\begin{rem}
The metric of any lightlike parahyperK\"ahler Heisenberg four-manifold in terms is isometric to:
\begin{equation*}
g=-\d t^2+ \fre^u_{t_0} \odot \fre^v_{t_0}+4k(t-t_0)   \fre^v_{t_0} \otimes  \fre^v_{t_0}+\fre^3_{t_0} \otimes \fre^3_{t_0}\,.
\end{equation*}  
We check by direct inspection that it is indeed flat. 
\end{rem}
\textbf{Proof of Theorem \ref{thm:allhk}}.  Gathering the results given in Propositions \ref{prop:hkriemsol}, \ref{prop:hkspacesol} and \ref{prop:hklightsol}, we conclude.

\end{document}